\documentclass[letterpaper, 10 pt, conference]{ieeeconf}  

\IEEEoverridecommandlockouts                              
\usepackage[table]{xcolor}
\usepackage{dsfont}
\usepackage{amsmath} 
\usepackage{amssymb,amsmath,graphicx,subfigure,times,float}
\newcommand{\norm}[1]{\ensuremath{\left\| #1 \right\|}}

\newcommand{\bracket}[1]{\ensuremath{\left[ #1 \right]}}

\newcommand{\refeqn}[1]{(\ref{eqn:#1})}
\newcommand{\reffig}[1]{Figure \ref{fig:#1}}
\newcommand{\tr}[1]{\mbox{tr}\ensuremath{\negthickspace\bracket{#1}}}
\newcommand{\trs}[1]{\mathrm{tr}\ensuremath{[#1]}}

\newcommand{\SO}{\ensuremath{\mathsf{SO(3)}}}

\newcommand{\SE}{\ensuremath{\mathsf{SE(3)}}}

\renewcommand{\Re}{\ensuremath{\mathbb{R}}}
\newcommand{\Sph}{\ensuremath{\mathsf{S}}}

\newcommand{\xb}{\mathbf{x}}

\newtheorem{prop}{Proposition}

\newcommand{\sat}{\mathrm{sat}}
\usepackage{hyperref}

\newcommand{\zb}{\mathbf{z}}

\title{\LARGE \bf
Extended Kalman Filter on SE(3) for Geometric Control of a Quadrotor UAV}

\author{Farhad A. Goodarzi and Taeyoung Lee$^{*}$
\thanks{Farhad A. Goodarzi and Taeyoung Lee, Mechanical and Aerospace Engineering, The George
Washington University, Washington DC 20052
{\tt\small {\{fgoodarzi,tylee}\}@gwu.edu}}%
\thanks{$^*$This research has been supported in part by NSF under the grants CMMI-1243000 (transferred from 1029551), CMMI-1335008, and CNS-1337722.}}

\begin{document}

\maketitle
\thispagestyle{empty}
\pagestyle{empty}

\begin{abstract}


An extended Kalman filter (EKF) is developed on the special Euclidean group, $\SE$ for geometric control of a quadrotor UAV. It is obtained by performing an extensive linearization on $\SE$ to estimate the state of the quadrotor from noisy measurements. Proposed estimator considers all the coupling effects between rotational and translational dynamics, and it is developed in a coordinate-free fashion. The desirable features of the proposed EKF are illustrated by numerical examples and experimental results for several scenarios. The proposed estimation scheme on $\SE$ has been unprecedented and these results can be particularly useful for aggressive maneuvers in GPS denied environments or in situations where parts of onboard sensors fail. 
\end{abstract}

\section{Introduction}
Quadrotor UAVs are being utilized for various missions such as Mars surface exploration, search and rescue, and payload transportation~\cite{kumargrasp2013,kim2013grasping,gooddaewontaeyoungacc14} due to their simple structure and outstanding capabilities. It is very important to implement a robust controller which can handle uncertainties and disturbances to ensure the safety during mission. Different sensors utilized for controlling purposes commonly provide noisy measurements, which can cause instability and may result in the failure of the mission. Also, there are some states,  such as translational velocity, that cannot be measured directly. Furthermore, there are cases where each sensor onboard may fail. Thus, a filter or estimator is required in real-time flight to ensure the safety from noisy measurements.

The critical issue in designing controllers and estimators for quadrotors is that they are mostly based on local coordinates. Some developments and estimation derivations are demonstrated at~\cite{05983144, 07347700,07364907,07361930} based on Euler angles. They involve complicated expressions for trigonometric functions, and they exhibit singularities in representing quadrotor attitudes, thereby restricting their ability to achieve complex rotational maneuvers significantly. Furthermore, most of the studies in the existing literature do not consider the coupling effect between translational and rotational dynamics explicitly, so the estimation process and controller design are presented for only attitude of the quadrotor~\cite{05649111,06390926}. 

On the other hand, this issue restricts the researchers to design and introduce an estimator that can perform in situation where an onboard sensor fail to work. For instance, in a GPS denied environment, where there is no direct measurements are available for the position and translational velocities, estimator fails to provide estimations since the translational and rotational dynamics are not coupled. A quaternion-based Kalman filter for real-time pose estimation was shown in~\cite{07053092}. Quaternions do not have singularities but, as the three-sphere double-covers the special orthogonal group, one attitude may be represented by two antipodal points on the three-sphere. This ambiguity should be carefully resolved in quaternion-based attitude control systems and estimator, otherwise they may exhibit unwinding, where a rigid body unnecessarily rotates through a large angle even if the initial attitude error is small~\cite{BhaBerSCL00}. References \cite{06713891,07364907} provide a Kalman filter based on Euler-angles for attitude and position control of the quadrotors without experimental results. Kalman filters may not work properly in a real-time experimental testbed due to sensitivity and several failure situations. It is very important to perform and provide experiments to validate the numerical and analytical solutions.

Recently, the dynamics of a quadrotor UAV is globally expressed on the special Euclidean group, $\SE$, and nonlinear control systems are developed to track outputs of several flight modes and complicated models for autonomous payload transportation with a large number of degrees of freedom~\cite{farhadacc15,IJCASPAPERfarhad14,farhadthesisphd}. Several aggressive maneuvers of a quadrotor UAV are demonstrated based on a hybrid control architecture. As they are directly developed on the special Euclidean group, complexities, singularities, and ambiguities associated with minimal attitude representations or quaternions are completely avoided. In this paper, the geometric nonlinear controller presented in the prior work of the authors in~\cite{Farhad2013,farhadASME15} is utilized with an EKF, developed for the closed loop system on $\SE$. A coordinate-free from of linearization is performed for the closed loop system to be used in the EKF.

In short, new contributions and the unique features of the dynamics model, control system, and the extended kalman filter proposed in this paper compared with other studies are as follows: (i) it is developed for the full dynamic model of a quadrotor UAVs on $\SE$, including the coupling effects between the translational dynamics and the rotational dynamics on a nonlinear manifold, (ii) the control systems are developed directly on the nonlinear configuration manifold in a coordinate-free fashion in the presence of unstructured uncertainties in both rotational and translational dynamics, (iii) a precise linearization for closed loop system and an extended kalman filter for closed loop system developed on $\SE$ to estimate the unmeasured states and improve the noisy measurements. Thus, singularities of local parameterization are completely avoided to generate agile maneuvers in a uniform way, (iv) the proposed algorithm is validated with numerical simulations along with real-time experiments with a quadrotor UAV.\\
This paper is organized as follows. A dynamic model is presented and problem is formulated at Section \ref{sec:DM}. Control systems are constructed at Sections \ref{sec:control} and extended kalman filter developments are presented in \ref{sec:Kalman}, which are followed by numerical examples in Section \ref{sec:NS} and an experimental results in Section \ref{sec:EXP}.
\section{Quadrotor's Dynamical Model}\label{sec:DM}

Consider a quadrotor UAV model illustrated in \reffig{QM}. We choose an inertial reference frame $\{\vec e_1,\vec e_2,\vec e_3\}$ and a body-fixed frame $\{\vec b_1,\vec b_2,\vec b_3\}$. The origin of the body-fixed frame is located at the center of mass of this vehicle. The first and the second axes of the body-fixed frame, $\vec b_1,\vec b_2$, lie in the plane defined by the centers of the four rotors.

The configuration of this quadrotor UAV is defined by the location of the center of mass and the attitude with respect to the inertial frame. Therefore, the configuration manifold is the special Euclidean group $\SE$, which is the semi-direct product of $\Re^3$ and the special orthogonal group $\SO=\{R\in\Re^{3\times 3}\,|\, R^TR=I,\, \det{R}=1\}$. 

The mass and the inertial matrix of a quadrotor UAV are denoted by $m\in\Re$ and $J\in\Re^{3\times 3}$. Its attitude, angular velocity, position, and velocity are defined by $R\in\SO$, $\Omega,x,v\in\Re^3$, respectively, where the rotation matrix $R$ represents the linear transformation of a vector from the body-fixed frame to the inertial frame and the angular velocity $\Omega$ is represented with respect to the body-fixed frame. The distance between the center of mass to the center of each rotor is $d\in\Re$, and the $i$-th rotor generates a thrust $f_i$ and a reaction torque $\tau_i$ along $-\vec b_3$ for $1\leq i \leq 4$. The magnitude of the total thrust and the total moment in the body-fixed frame are denoted by $f\in\Re$, $M\in\Re^3$, respectively. 
\begin{figure}[h]
\setlength{\unitlength}{0.8\columnwidth}\footnotesize
\centerline{
\begin{picture}(1.2,0.91)(0,0)
\put(0,0){\includegraphics[width=0.95\columnwidth]{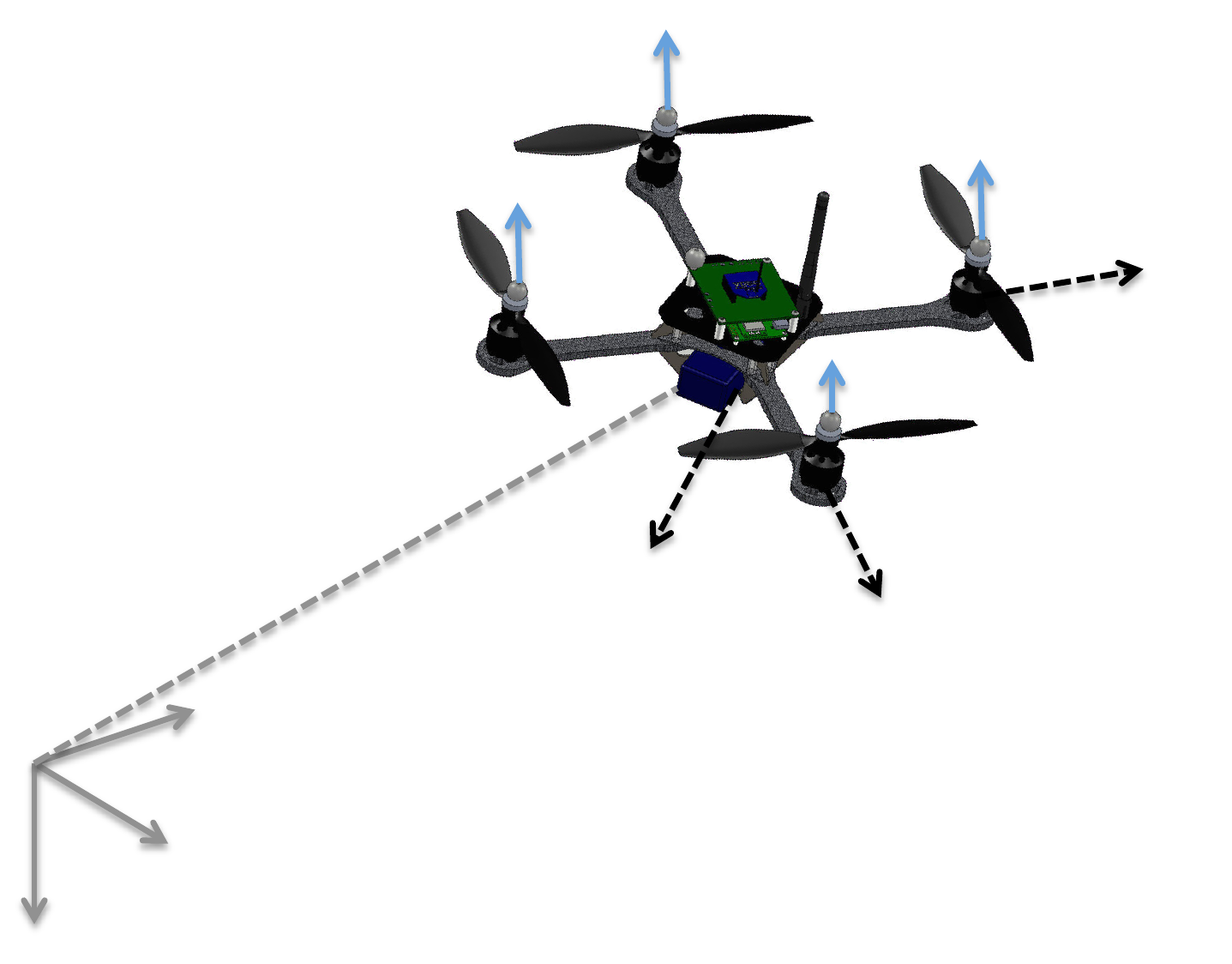}}
\put(0.2,0.23){\shortstack[c]{$\vec e_1$}}
\put(0.17,0.08){\shortstack[c]{$\vec e_2$}}
\put(0.02,0.0){\shortstack[c]{$\vec e_3$}}
\put(1.13,0.66){\shortstack[c]{$\vec b_1$}}
\put(0.85,0.3){\shortstack[c]{$\vec b_2$}}
\put(0.61,0.34){\shortstack[c]{$\vec b_3$}}
\put(0.92,0.8){\shortstack[c]{$f_1$}}
\put(0.58,0.86){\shortstack[c]{$f_2$}}
\put(0.52,0.68){\shortstack[c]{$f_3$}}
\put(0.83,0.55){\shortstack[c]{$f_4$}}
\put(0.35,0.41){\shortstack[c]{$x$}}
\end{picture}}
\caption{Quadrotor model}\label{fig:QM}
\end{figure}
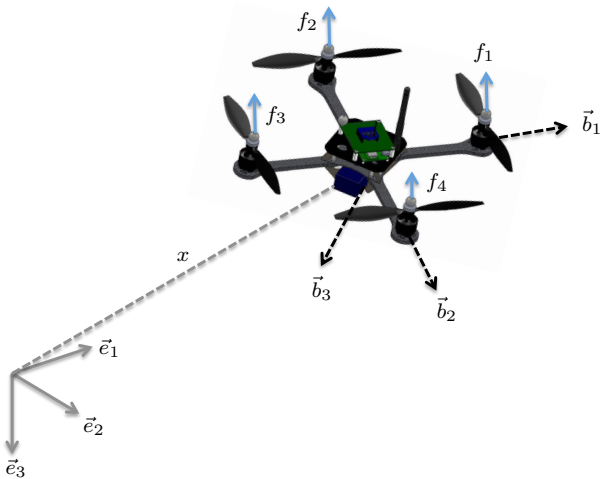
By the definition of the rotation matrix $R\in\SO$, the direction of the $i$-th body-fixed axis $\vec b_i$ is given by $Re_i$ in the inertial frame, where $e_1=[1;0;0],e_2=[0;1;0],e_3=[0;0;1]\in\Re^3$. Therefore, the total thrust vector is given by $-fRe_3\in\Re^3$ in the inertial frame. In this paper, the thrust magnitude $f\in\Re$ and the moment vector $M\in\Re^3$ are viewed as control inputs. The corresponding equations of motion are given by
\begin{gather}
\dot x  = v,\label{eqn:EL1}\\
m \dot v = mge_3 - f R e_3 + \Delta_x,\label{eqn:EL2}\\
\dot R = R\hat\Omega,\label{eqn:EL3}\\
J\dot \Omega + \Omega\times J\Omega = M + \Delta_R,\label{eqn:EL4}
\end{gather}
where the \textit{hat map} $\hat\cdot:\Re^3\rightarrow\SO$ is defined by the condition that $\hat x y=x\times y$ for all $x,y\in\Re^3$. This identifies the Lie algebra $\SO$ with $\Re^3$ using the vector cross product in $\Re^3$. The inverse of the hat map is denoted by the \textit{vee} map, $\vee:\SO\rightarrow\Re^3$. Throughout this paper, the two-norm of a matrix $A$ is denoted by $\|A\|$. The standard dot product in $\Re^n$ is denoted by $\cdot$, i.e., $x\cdot y =x^Ty$ for any $x,y\in\Re^n$. Unstructured, but fixed uncertainties in the translational dynamics and the rotational dynamics of a quadrotor UAV are denoted by $\Delta_x$ and $\Delta_R\in\Re^3$, respectively. 

\section{Geometric Nonlinear Controller}\label{sec:control}
Since the quadrotor UAV has four inputs, it is possible to achieve asymptotic output tracking for at most four quadrotor UAV outputs.    The quadrotor UAV has three translational and three rotational degrees of freedom; it is not possible to achieve asymptotic output tracking of both attitude and position of the quadrotor UAV. This  motivates us to introduce two flight modes, namely (1) an attitude controlled flight mode, and (2) a position controlled flight mode. While a quadrotor UAV is under-actuated, a complex flight maneuver can be defined by specifying a concatenation of flight modes together with conditions for switching between them.
\subsection{Attitude Tracking Errors}
Suppose that a smooth attitude command $R_d(t)\in\SO$ satisfying the following kinematic equation is given:
\begin{align}
\dot R_d = R_d \hat\Omega_d,
\end{align}
where $\Omega_d(t)\in\Re^3$ is the desired angular velocity, which is assumed to be uniformly bounded. We first define errors associated with the attitude dynamics as follows~\cite{BulLew05,LeeITCST13}.
\begin{prop}\label{prop:1}
For a given tracking command $(R_d,\Omega_d)$, and the current attitude and angular velocity $(R,\Omega)$, we define an attitude error function $\Psi:\SO\times\SO\rightarrow\Re$, an attitude error vector $e_R\in\Re^3$, and an angular velocity error vector $e_\Omega\in \Re^3$ as follows~\cite{LeeITCST13}:
\begin{gather}
\Psi (R,R_d) = \frac{1}{2}\tr{I-R_d^TR},\label{eqn:psii}\\
e_R =\frac{1}{2} (R_d^TR-R^TR_d)^\vee,\label{eqn:errr}\\
e_\Omega = \Omega - R^T R_d\Omega_d\label{eqn:eomega}.
\end{gather}
\end{prop}
\begin{proof} 
See \cite{farhadASME15}.
\end{proof}
\subsection{Attitude Tracking Controller}
We now introduce a nonlinear controller for the attitude controlled flight mode:
\begin{align}
M & = -k_R e_R -k_\Omega e_\Omega -k_I e_I\nonumber\\
&\qquad +(R^TR_d\Omega_d)^\wedge J R^T R_d \Omega_d + J R^T R_d\dot\Omega_d,\label{eqn:aM}\\
e_I & = \int_0^t e_\Omega(\tau)+ c_2e_R(\tau)d\tau,\label{eqn:eI}
\end{align}
where $k_R,k_\Omega,k_I,c_2$ are positive constants. The control moment is composed of proportional, derivative, and integral terms, augmented with additional terms to cancel out the angular acceleration caused by the desired angular velocity. Unlike common integral control terms where the attitude error is integrated only, here the angular velocity error is also integrated at \refeqn{eI}. This unique term is required to show exponential stability in the presence of the disturbance $\Delta_R$ in the subsequent analysis. 
\begin{prop}{(Attitude Controlled Flight Mode)}\label{prop:Att}
Consider the control moment $M$ defined in \refeqn{aM}. For positive constants $k_R,k_\Omega$, the zero equilibrium of tracking errors and the estimation errors is stable in the sense of Lyapunov, and $e_R,e_\Omega\rightarrow 0$ as $t\rightarrow\infty$. 
\end{prop}
\begin{proof}
See \cite{farhadASME15}.
\end{proof}
While these results are developed for the attitude dynamics of a quadrotor UAV, they can be applied to the attitude dynamics of any rigid body. Nonlinear adaptive controllers have been developed for attitude stabilization in terms of modified Rodriguez parameters~\cite{SubJAS04} and quaternions~\cite{SubAkeJGCD04}, and for attitude tracking in terms of Euler-angles~\cite{ShoJuaPACC02}. The proposed tracking control system is developed on $\SO$, therefore it avoids singularities of Euler-angles and Rodriguez parameters, as well as unwinding of quaternions. 

\subsection{Position Tracking Errors}
We introduce a nonlinear controller for the position controlled flight mode in this section. Suppose that an arbitrary smooth position tracking command $x_d (t) \in \Re^3$ is given. The position tracking errors for the position and the velocity are given by:
\begin{align}
e_x  = x - x_d,\quad
e_v  = \dot e_x = v - \dot x_d.
\end{align}
Similar with \refeqn{eI}, an integral control term for the position tracking controller is defined as
\begin{align}
e_i = \int_0^t e_v(\tau) + c_1 e_x (\tau) d\tau,
\end{align}
for a positive constant $c_1$ specified later. For a positive constant $\sigma\in\Re$, a saturation function $\sat_\sigma:\Re\rightarrow [-\sigma,\sigma]$ is introduced as
\begin{align*}
\sat_\sigma(y) = \begin{cases}
\sigma & \mbox{if } y >\sigma\\
y & \mbox{if } -\sigma \leq y \leq\sigma\\
-\sigma & \mbox{if } y <-\sigma\\
\end{cases}.
\end{align*}
If the input is a vector $y\in\Re^n$, then the above saturation function is applied element by element to define a saturation function $\sat_\sigma(y):\Re^n\rightarrow [-\sigma,\sigma]^n$ for a vector.
In the position controlled tracking mode, the attitude dynamics is controlled to follow the computed attitude $R_c(t)\in\SO$ and the computed angular velocity $\Omega_c(t)$ defined as
\begin{align}
R_c=[ b_{1_c};\, b_{3_c}\times b_{1_c};\, b_{3_c}],\quad \hat\Omega_c = R_c^T \dot R_c\label{eqn:RdWc},
\end{align}
where $b_{3_c}\in\Sph^2$ is given by
\begin{align}
 b_{3_c} = -\frac{-k_x e_x - k_v e_v -k_i\sat_\sigma(e_i) - mg e_3 +m\ddot x_d}{\norm{-k_x e_x - k_v e_v -k_i\sat_\sigma(e_i)- mg e_3 + m\ddot x_d}},\label{eqn:Rd3}
\end{align}
for positive constants $k_x,k_v,k_i,\sigma$. The unit vector $b_{1_c}\in\Sph^2$ is selected to be orthogonal to $b_{3_c}$, thereby guaranteeing that $R_c\in\SO$. It can be chosen to specify the desired heading direction, and the detailed procedure to select $b_{1c}$ is described later at Section \ref{sec:b1c}. 
\subsection{Position Tracking Controller}
The nonlinear controller for the position controlled flight mode, described by control expressions for the  thrust magnitude and the  moment vector, are:
\begin{align}
f & = ( k_x e_x + k_v e_v +k_i\sat_\sigma(e_i)+ mg e_3-m\ddot x_d)\cdot Re_3,\label{eqn:f}\\
M & = -k_R e_R -k_\Omega e_\Omega -k_I e_I\nonumber\\
&\qquad +(R^TR_c\Omega_c)^\wedge J R^T R_c \Omega_c + J R^T R_c\dot\Omega_c.\label{eqn:M}
\end{align}
The nonlinear controller given by Eq.~\refeqn{f}, \refeqn{M} can be given a backstepping interpretation. The computed attitude $R_c$ given in Eq.~\refeqn{RdWc} is selected so that the thrust axis $-b_3$ of the quadrotor UAV tracks the computed direction given by $-b_{3_c}$ in \refeqn{Rd3}, which is a direction of the thrust vector that achieves position tracking. The moment expression \refeqn{M} causes the attitude of the quadrotor UAV to asymptotically track $R_c$ and the thrust magnitude expression \refeqn{f} achieves asymptotic position tracking. 
\begin{prop}{(Position Controlled Flight Mode)}\label{prop:Pos}
Suppose that the initial conditions satisfy
\begin{align}\label{eqn:Psi0}
\Psi(R(0),R_c(0)) < \psi_1 < 1,
\end{align}
for positive constant $\psi_1$. Consider the control inputs $f,M$ defined in \refeqn{f}-\refeqn{M}.
For positive constants $k_x,k_v$, the zero equilibrium of the tracking errors and the estimation errors is stable in the sense of Lyapunov and all of the tracking error variables asymptotically converge to zero. Also, the estimation errors are uniformly bounded.
\end{prop}
\begin{proof}
See \cite{farhadASME15}.
\end{proof}
\begin{prop}{(Position Controlled Flight Mode with a Larger Initial Attitude Error)}\label{prop:Pos2}
Suppose that the initial conditions satisfy
\begin{gather}
1\leq \Psi(R(0),R_c(0)) < 2\label{eqn:eRb3},\quad \|e_x(0)\| < e_{x_{\max}},
\end{gather}
for a constant $e_{x_{\max}}$. Consider the control inputs $f,M$ defined in \refeqn{f}-\refeqn{M}, where the control parameters satisfy \refeqn{Psi0} for a positive constant $\psi_1<1$. Then the zero equilibrium of the tracking errors is attractive, i.e., $e_x,e_v,e_R,e_\Omega\rightarrow 0$ as $t\rightarrow\infty$. 
\end{prop}
\begin{proof}
See \cite{farhadASME15}.
\end{proof}
\subsection{Direction of the First Body-Fixed Axis}\label{sec:b1c}
As described above, the construction of the orthogonal matrix $R_c$ involves having its third column $b_{3_c}$ specified by \refeqn{Rd3}, and its first column $b_{1_c}$ is arbitrarily chosen to be orthogonal to the third column, which corresponds to a one-dimensional degree of choice. By choosing $b_{1_c}$ properly, we constrain the asymptotic direction of the first body-fixed axis. This can be used to specify the heading direction of a quadrotor UAV in the horizontal plane~\cite{LeeLeoAJC13}.


\section{Extended Kalman Filter (EKF)}\label{sec:Kalman}
The Kalman Filter is an algorithm which utilizes a set of measurements observed over time while considering statistical noise, and generates estimates of uncertain variables that are more accurate and precise than those variables which are based on a single measurement. It contains of two general steps, prediction (or flow update), and the measurement update. It uses the prior knowledge of state along with the measurements to predict and update the state variables. The EKF is motivated by linearizing the nonlinear system considering the control input as the closed loop system. In the following, we presented the nonlinear equations used in the EKF derivations, along with a precise linearization on $\SE$, followed by the derivations for EKF implementation on the system.

\subsection{Flow Update}
The equations of motion, given by \refeqn{EL1}-\refeqn{EL4}, can be rearranged as
\begin{align}\label{eqn:NLEM}
\dot{\chi}=f(\chi,u)+w,
\end{align}
where the state vector $\chi\in\Re^{24\times1}$ is given by
\begin{align}
\chi=[x,v,R,\Omega,e_{i1},e_{i1}]^{T},
\end{align}
and the process noise is denoted by $w\in\Re^{24}$. Assume the covariance of the process noise is $Q=\text{E}[ww^{T}]\in\Re^{24\times 24}$. In the above nonlinear equation of motion, the translational and rotational dynamics are coupled as rotation matrix $R$ and angular velocity $\Omega$ directly appear in the dynamics of the rigid body.

Let the measurement $z\in\Re^{24}$ be a nonlinear function of state
\begin{align}
z=h(\chi)+v,
\end{align}
where $v\in\Re^p$ denotes measurement noise with covariance $\mathcal R=\text{E}[vv^{T}]$.

Next, we show the linearized equations of motion. The key idea is to represent the infinitesimal variation of $R\in\SO$ using the exponential map
\begin{align}
\delta R = \frac{d}{d\epsilon}\bigg|_{\epsilon = 0} R \exp (\epsilon \hat\eta) = R\hat\eta,\label{eqn:delR}
\end{align}
for $\eta\in\Re^3$. The unit vector $\frac{\eta}{\|\eta\|}$ corresponds to the axis of the rotation, expressed in the body-fixed frame, and $\norm{\eta}$ is the rotation angle. 

Using this, we linearize the controlled equation of motion. 
\begin{prop}\label{prop:linearizedcontroled}
Consider the controlled dynamical model presented in~\refeqn{EL1}, \refeqn{EL2}, \refeqn{EL3}, and \refeqn{EL4}, along with the control inputs presented in~\refeqn{f} and \refeqn{M}. The linearized equations of motion for the closed loop dynamics are given by 
\begin{align}\label{eqn:EOMLin}
\delta\dot{\xb}=A_{L}\delta\xb,
\end{align}
where $\delta$ represents an infinitesimal change of the state vector $\xb\in\Re^{18\times1}$ given by
\begin{align}
\xb=\begin{bmatrix}
{x},v,\eta,\Omega,{e}_{i1},{e}_{i2}
\end{bmatrix}^{T},\label{eqn:xb}
\end{align}
and $A_{L}\in\Re^{18\times 18}$ is given in details in the appendix.
\end{prop}
\begin{proof}
See Appendix.
\end{proof}

Next, we describe the proposed extended Kalman filter.  Let $\bar{\;}$ denote an estimate of state and assume that the mean of the initial state are given by $\bar{\chi}_0$. The initial state for the linearized system, namely $\bar{\xb}_0$ can be obtained from \refeqn{delR}, and suppose the variance of $\bar{\xb}_0$ is given by $\mathcal{P}_0\in\Re^{18\times 18}$.


The extended Kalman filter is formulated in the discrete-time setting, and let subscript $k$ denote the value of a variable at the $k$-th time step. Also, superscripts $-$ and $+$ mean the a priori (before measurements) and a posteriori (after measurements) variables,  respectively.

Runge-Kutta method~\cite{RungeKutta87} is utilized for \refeqn{NLEM} to propagate the mean from the current a posteriori estimate $\bar{\chi}_{k}^+$ to the next a priori estimate $\bar{\chi}_{k+1}^-$. The a priori state covariance uses the linearized matrix $A_{L_k}$,
\begin{align}
\mathcal{P}_{k+1}^- = A_{L_k} \mathcal{P}_k^+ A_{L_k}^T + \mathcal{Q}_k,\label{eqn:Pkpm}
\end{align}
where $A_{L_k}$ serves to reshape the covariance and $\mathcal{Q}_k$ is the added uncertainty along the propagation.

\subsection{Measurement Update}
The measurement matrix is linearized with respect to the state to obtain
\begin{align}
H_{k+1}=\frac{\partial h}{\partial \chi}\bigg|_{x=\bar{\chi}_{k+1}^-}\in\Re^{24\times 18},
\end{align}
where the derivative with respect to $R$ is taken by using \refeqn{delR}.

The predicted knowledge $(\bar{\chi}_{k+1}^-,\mathcal{P}_{k+1}^-)$ is converted into the reduced state $(\bar{\xb}_{k+1}^-,\mathcal{P}_{k+1}^-)$ from \refeqn{delR} and \refeqn{xb}. Then, the measurement update provides a correction to obtain an \textit{a posteriori} estimate $(\bar{\xb}_{k+1}^+,\mathcal{P}_{k+1}^+)$:
\begin{gather}
\mathcal{K}_{k+1} = \mathcal{P}_{k+1}^- H_{k+1}^T(H_{k+1} \mathcal{P}_{k+1}^- H_{k+1}^T + \mathcal{R}_{k+1})^{-1}\label{eqn:K}\\
\bar \xb^+_{k+1} = \bar \xb^-_{k+1} + \mathcal{K}_{k+1} ( \zb_{k+1}-h(\bar \xb^-_{k+1})),\label{eqn:xhatp}\\
\mathcal{P}^+_{k+1} = (I-\mathcal{K}_{k+1} H_{k+1} )\mathcal{P}^{-}_{k+1},\label{eqn:Pp}
\end{gather}
where the updated state $\bar \xb^+_{k+1}$ takes both the predicted process and measurement into account and the matrix $\mathcal{K}_{k+1}$ is referred to as \textit{Kalman gain} at the $(k+1)$-th time step. Also, $\zb_{k+1}$ is the measurement vector obtained from the reference attitude using Eq.~\refeqn{delR}. The key idea of implementing the attitude in the measurement update, is that the difference between the predicted attitude and the measured attitude is given by $\eta$ in Eq.~\refeqn{xhatp} for $\zb_{k+1}$. Finally, the $\bar \chi_{k+1}^+$ is obtained from $\bar \xb^+_{k+1}$ using Eq.~\refeqn{delR}.


\section{Numerical Simulations}\label{sec:NS}
We demonstrate the desirable properties of the proposed extended kalman filter with two numerical examples. Consider the geometric nonlinear control system designed in previous section, we need to provide the controller with state variables namely, position, velocity, attitude and angular velocity of the quadrotor. We implement the extended kalman filter to estimate the un-measured states and also improve the noisy measurements from sensors. For both examples, properties of the quadrotor are chosen as
\begin{gather}
m=0.755\,\mathrm{kg},\; J=\mathrm{diag}[0.557,\,0.557,\,1.05]\times 10^{-2}\,\mathrm{kgm^2},\nonumber
\end{gather}
and controller parameters are selected as follows: $k_x=13.84$, $k_v=4.84$, ${k_R}=0.67$, ${k_\Omega}= 0.11$, $k_{I}=0.01$, $B_{5}=B_{6}=0.1$. 
The following two fixed disturbances are included in this numerical example.
\begin{gather*}
\Delta_R=[0.01,-0.02,0.01]^T,\\
\Delta_x=[-0.02,0.01,-0.03]^T.
\end{gather*}
The initial conditions for the quadrotor are given by $x(0)=0_{3\times 1}$, $\dot{x}(0)=0_{3\times 1}$, $R(0)=I_{3\times 3}$, and $\Omega(0)=0_{3\times 1}$. Initial estimates of the attitude and angular velocity are given by $\bar{R}(0)=I_{3\times 3}$ and $\bar{\Omega}(0)=[0.1,-0.2,0.1]^{T}$ respectively for both examples.
\subsection{\textbf{Example 1- Estimating the translational velocity with large initial error and high-noisy measurements}}
In this example, we investigate a case where noisy measurements on position, attitude and angular velocity are available via sensors and we wish to estimate the translational velocity while improving the noisy measurements. The measurements noise covariance is chosen as $\mathcal{R}=1.0$ with process noise of $\mathcal{Q}=0.01$. Also the initial estimation values are chosen as $\bar{x}(0)=\bar{v}(0)=[4,4,-3]^{T}$ to have large initial estimation errors. Desired trajectory is chosen to be a Lissajous curve as
\begin{gather*}
x_d(t) =[\sin(t)+\frac{\pi}{2},\; \sin(2t),\; -0.5]^{T},\;b_{1d}=[1,\; 0,\; 0]^{T}.
\end{gather*}
and geometric nonlinear controller is employed along with proposed EKF to track the above desired trajectory. The observation matrix, $H\in\Re^{9\times18}$ utilized in this example is 
\begin{align*}
H=\begin{bmatrix}
I_{3}&0&0&0&0&0\\
0&0&I_{3}&0&0&0\\
0&0&0&I_{3}&0&0
\end{bmatrix}.
\end{align*}

\begin{figure}[h]
\centerline{
	        \subfigure[Quadrotor position, $x$, $\hat{x}$]{
		\includegraphics[width=0.5\columnwidth]{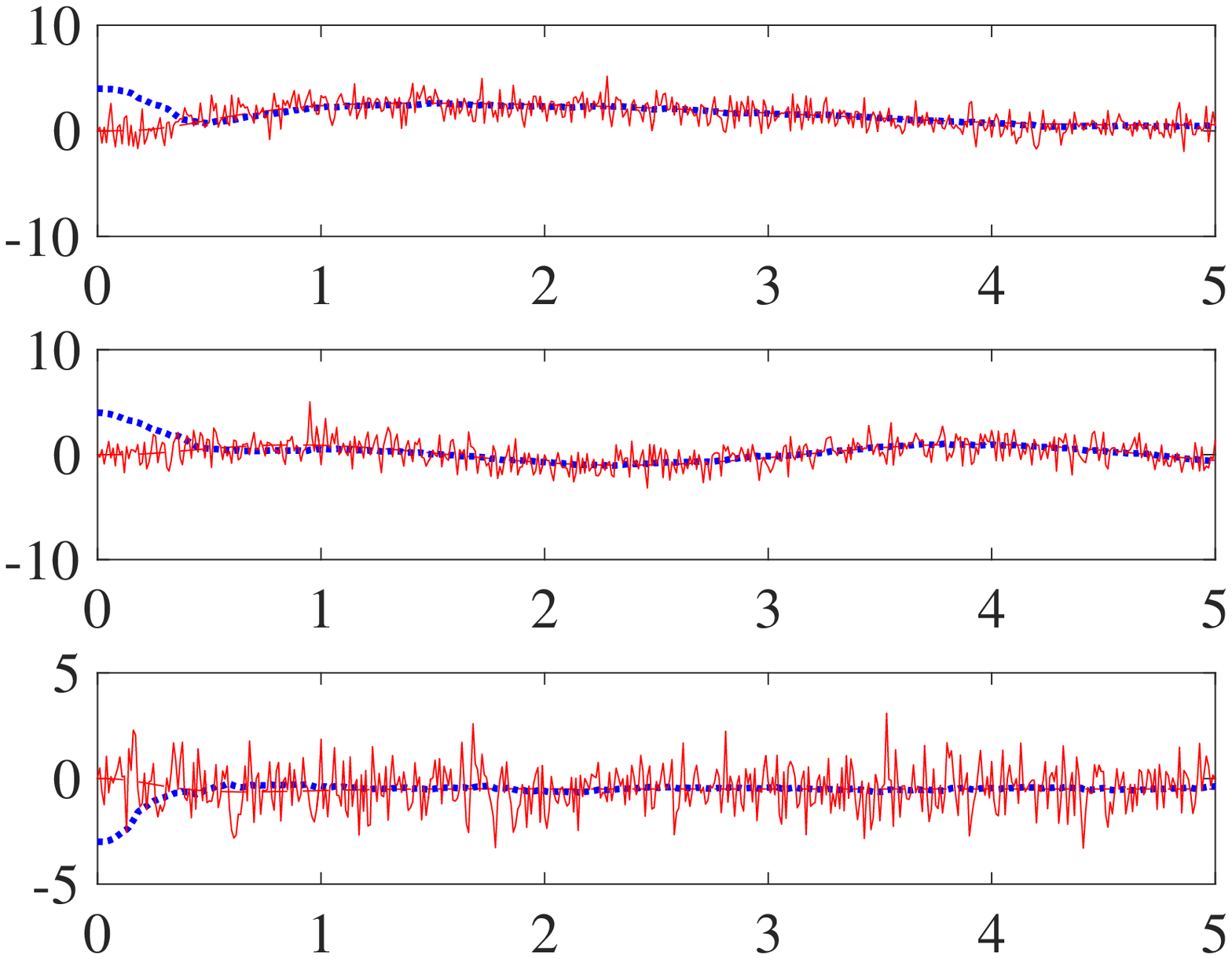}\label{fig:sim_x4}}
		\subfigure[$\Omega_{1}$, $\hat{\Omega}_{1}$]{
		\includegraphics[width=0.5\columnwidth]{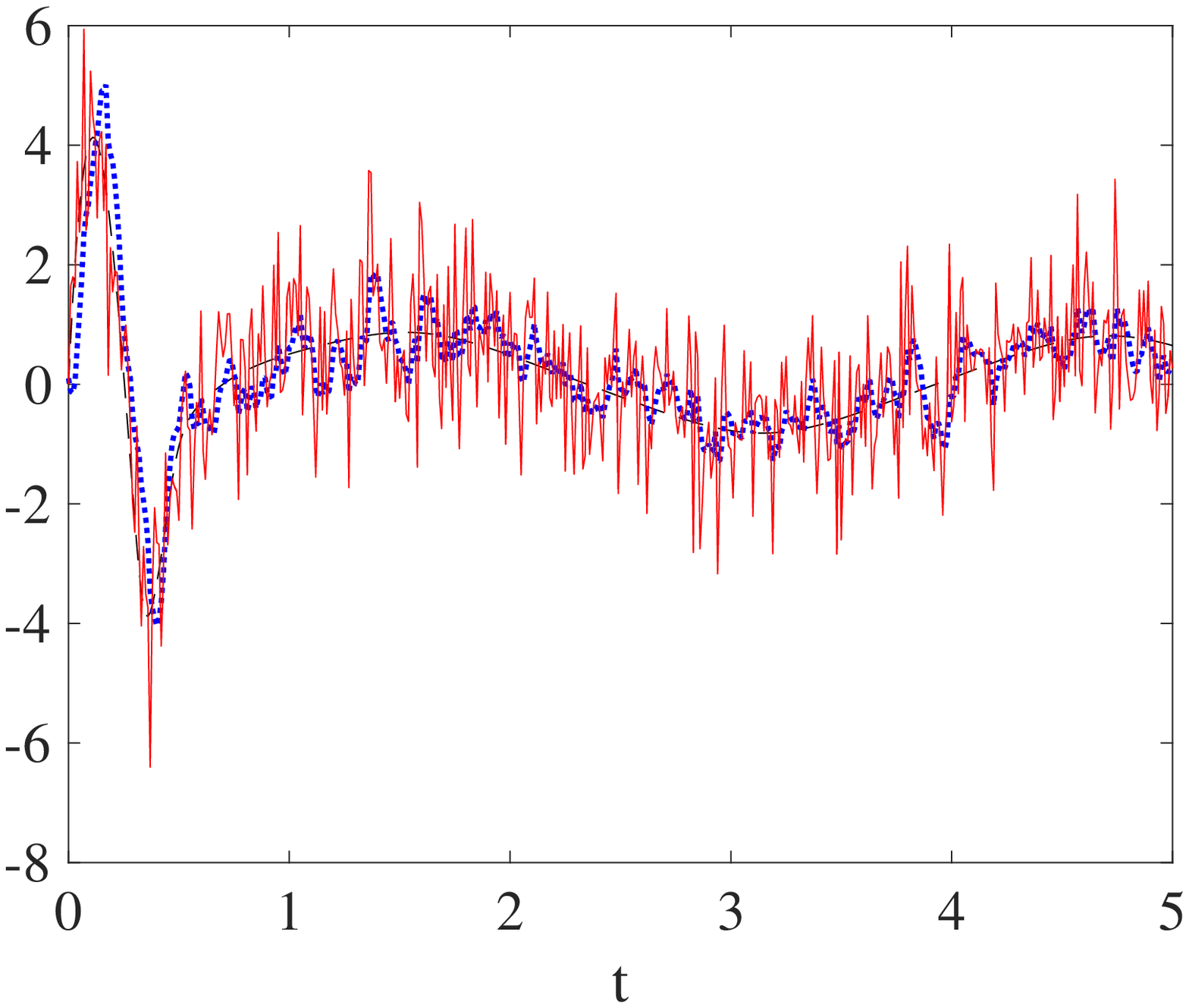}\label{fig:sim_W14}}}\centerline{
	        \subfigure[Quadrotor velocity, $v$, $\hat{v}$]{
		\includegraphics[width=0.5\columnwidth]{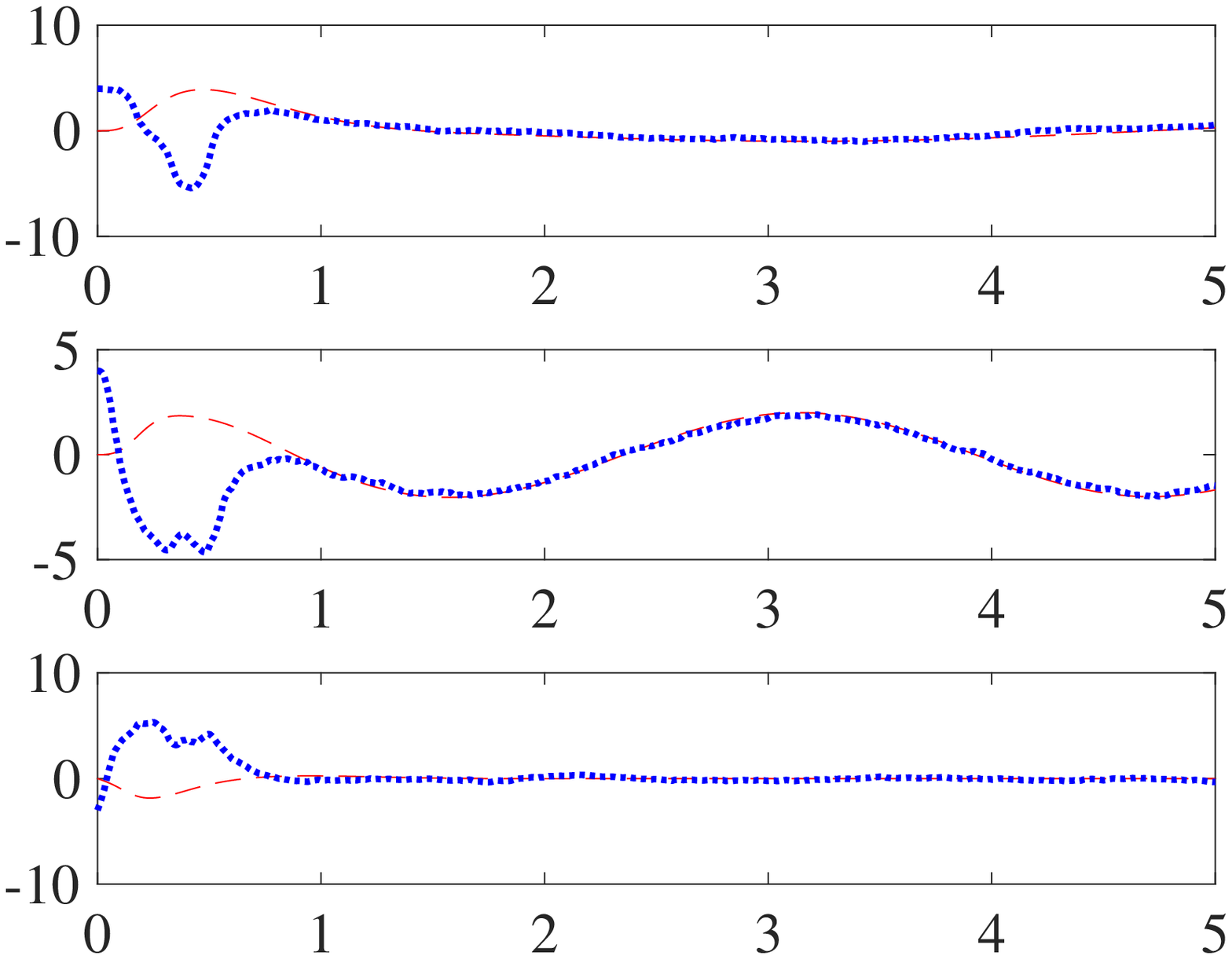}\label{fig:sim_v4}}
		\subfigure[$\Omega_{2}$, $\hat{\Omega}_{2}$]{
		\includegraphics[width=0.5\columnwidth]{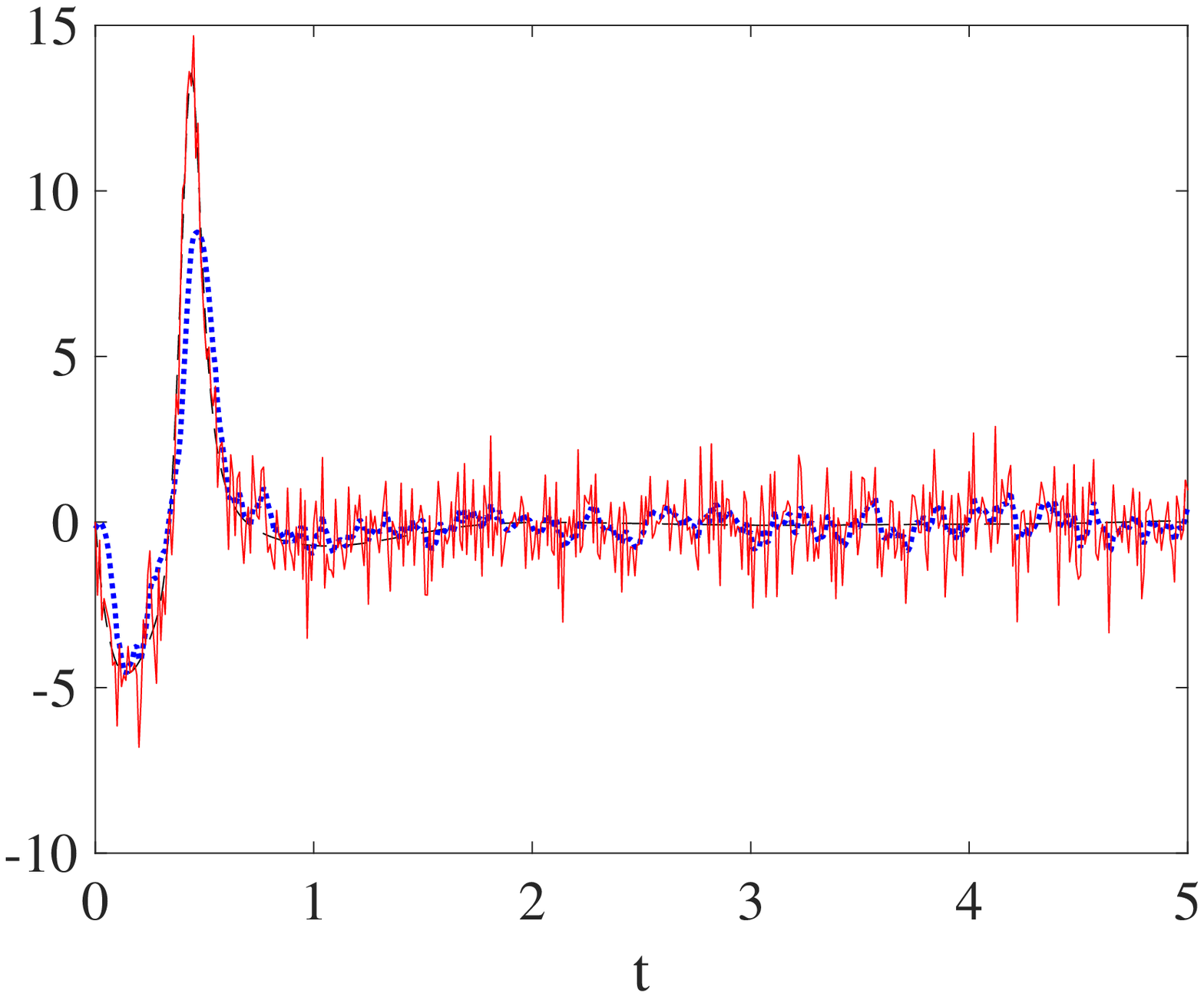}\label{fig:sim_W24}}}\centerline{
		\subfigure[3D view, $x$, $\hat{x}$]{
		\includegraphics[width=0.5\columnwidth]{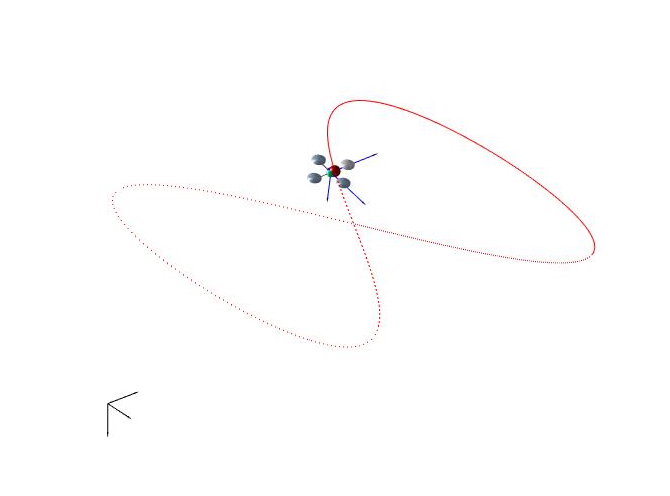}\label{fig:eq4}}			
		 \subfigure[$\Omega_{3}$, $\hat{\Omega}_{3}$]{
		\includegraphics[width=0.5\columnwidth]{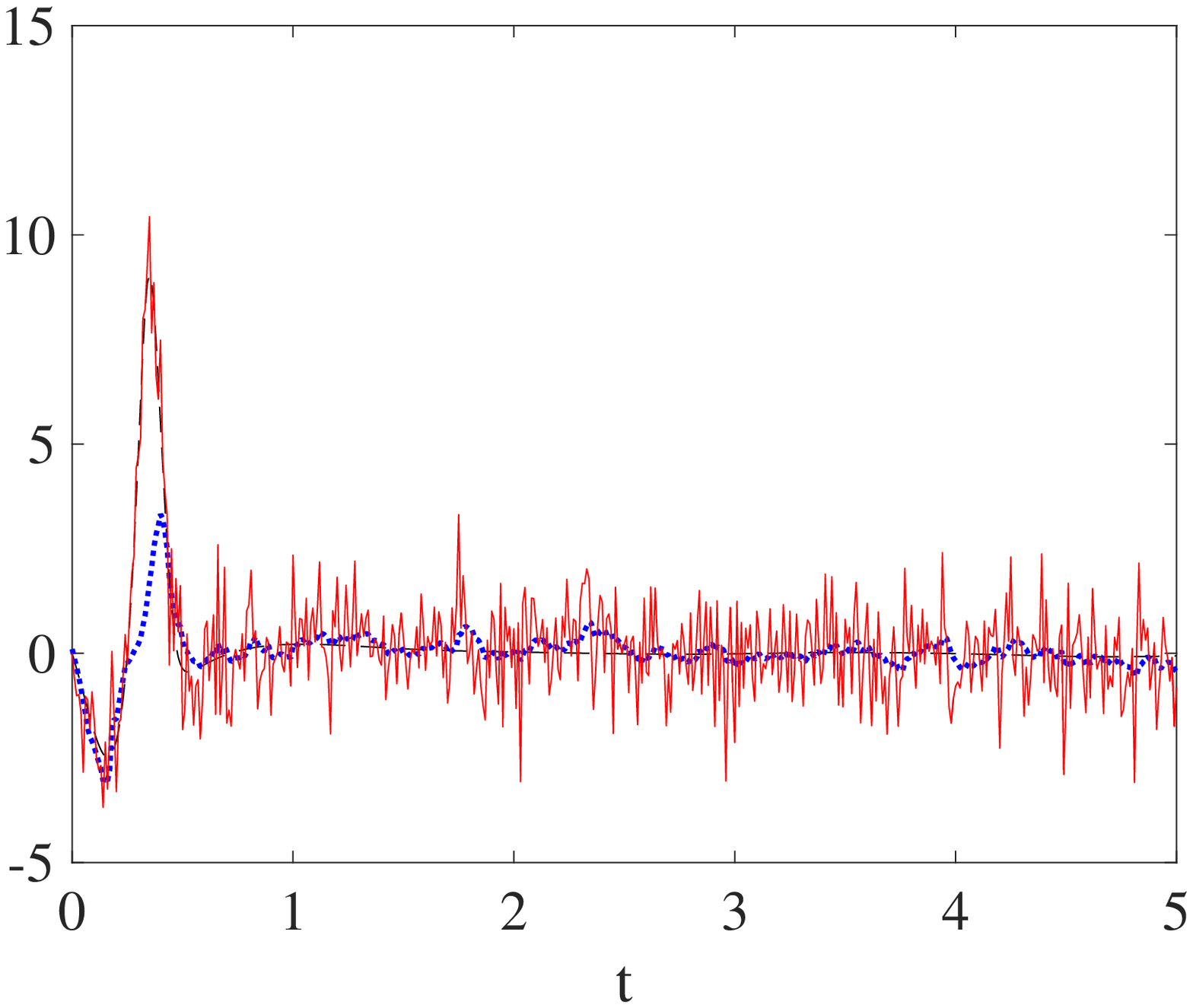}\label{fig:sim_W34}}}\centerline{
		\subfigure[Position error]{
		\includegraphics[width=0.5\columnwidth]{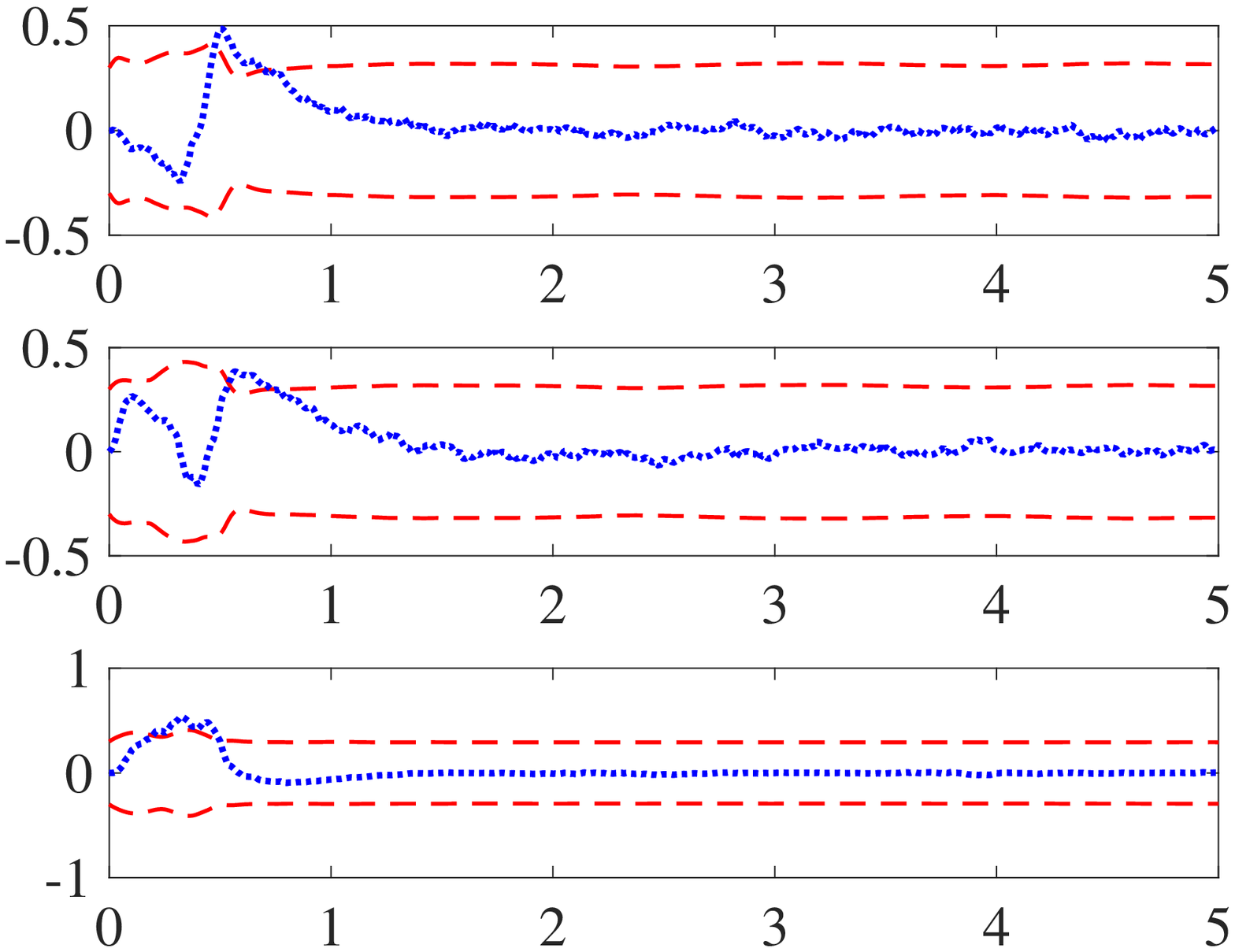}\label{fig:errorsx4}}
		\subfigure[Velocity error]{
		\includegraphics[width=0.5\columnwidth]{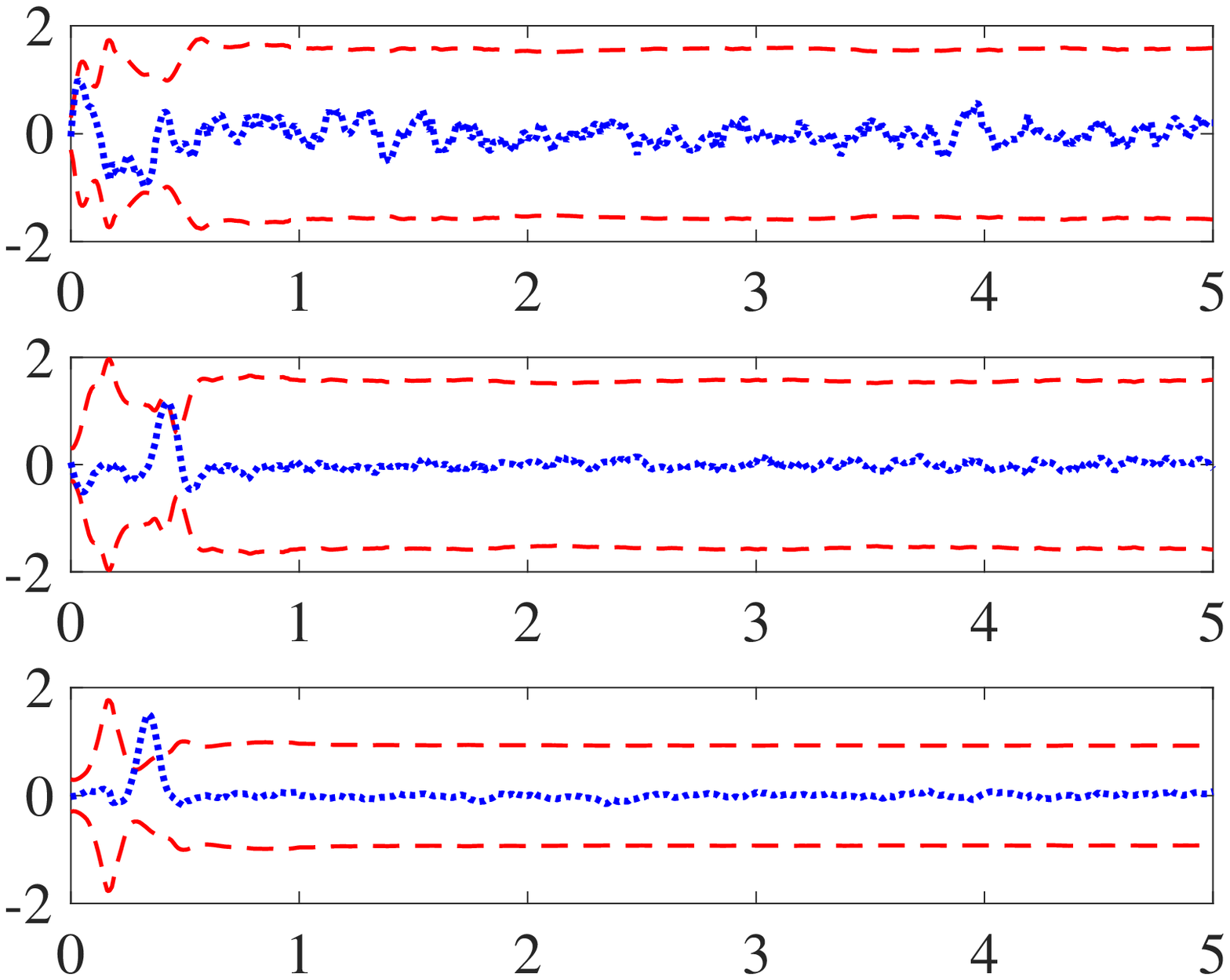}\label{fig:errorsv4}}
}
\caption{EKF performance in the first numerical example (dotted: EKF, dashed: desired, solid: noisy measurements). A short animation is also available at  \href{https://youtu.be/F4Vntws97RU}{https://youtu.be/F4Vntws97RU}}\label{fig:sim1}
\end{figure}

Figure \ref{fig:sim1} illustrates the performance of the proposed controller and EKF in this numerical simulation. Estimated position, $\bar{x}$ and angular velocity, $\Omega$ of the quadrotor is presented along with the noisy measurements at Figures \ref{fig:sim_x4}, \ref{fig:sim_W14}, \ref{fig:sim_W24}, and \ref{fig:sim_W34} respectively. Desired translational velocity and the estimated velocity outputted from proposed EKF is plotted at Figure \ref{fig:sim_v4} which shows s satisfactory estimate during the maneuver. Figures \ref{fig:errorsx4} and \ref{fig:errorsv4} are the position and velocity estimation errors calculated from the following expressions.
\begin{gather*}
\bar{e}_{x}=\bar{x}-x_{d},\;\bar{e}_{v}=\bar{v}-v_{d}.
\end{gather*}
\begin{figure}[h]
\centerline{
	        \subfigure[$R_{11}$]{
		\includegraphics[width=0.33\columnwidth]{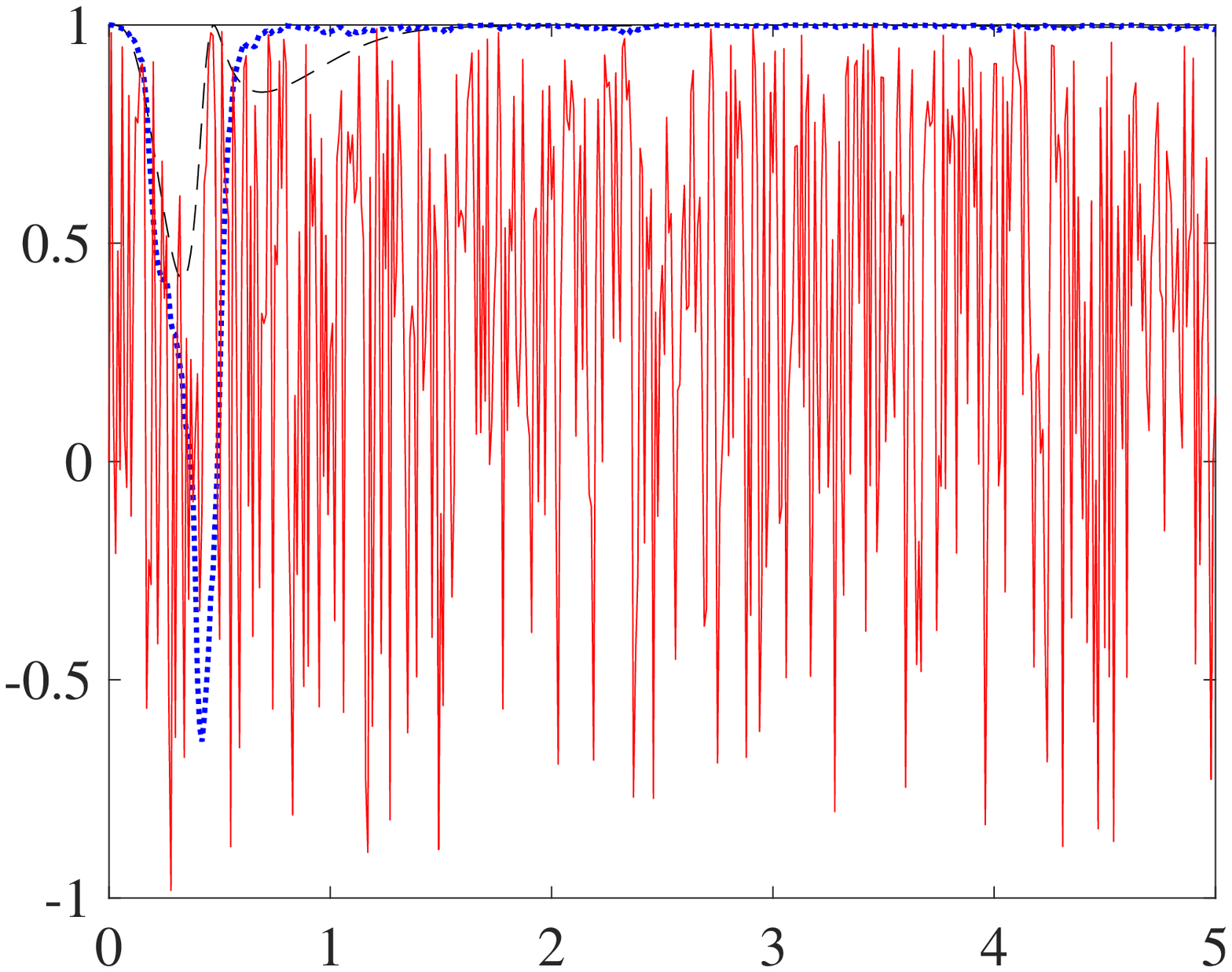}\label{fig:sim_R11_4}}\hspace{-0.3cm}
		\subfigure[$R_{12}$]{
		\includegraphics[width=0.33\columnwidth]{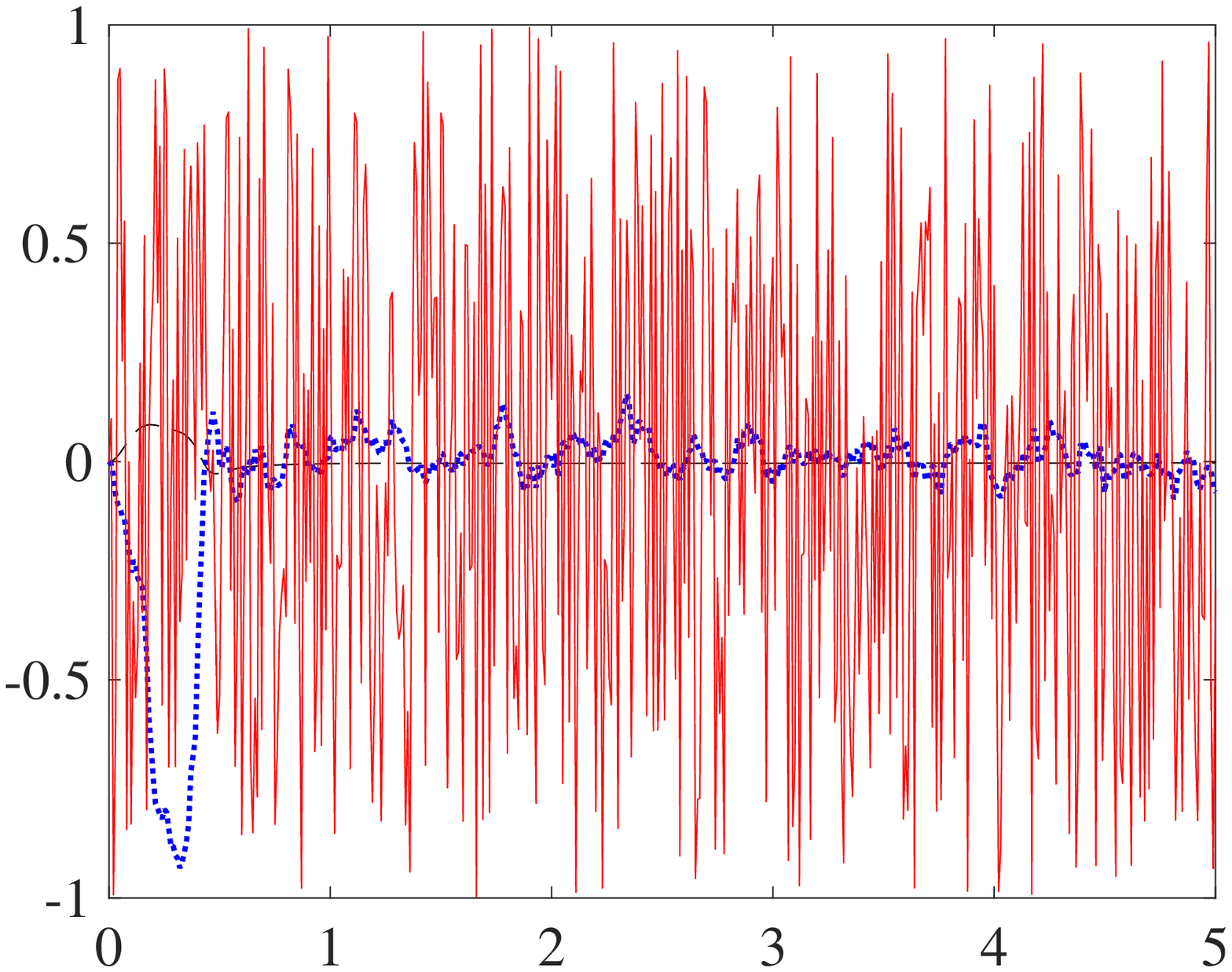}\label{fig:sim_R11_4}}\hspace{-0.3cm}
		\subfigure[$R_{13}$]{
		\includegraphics[width=0.33\columnwidth]{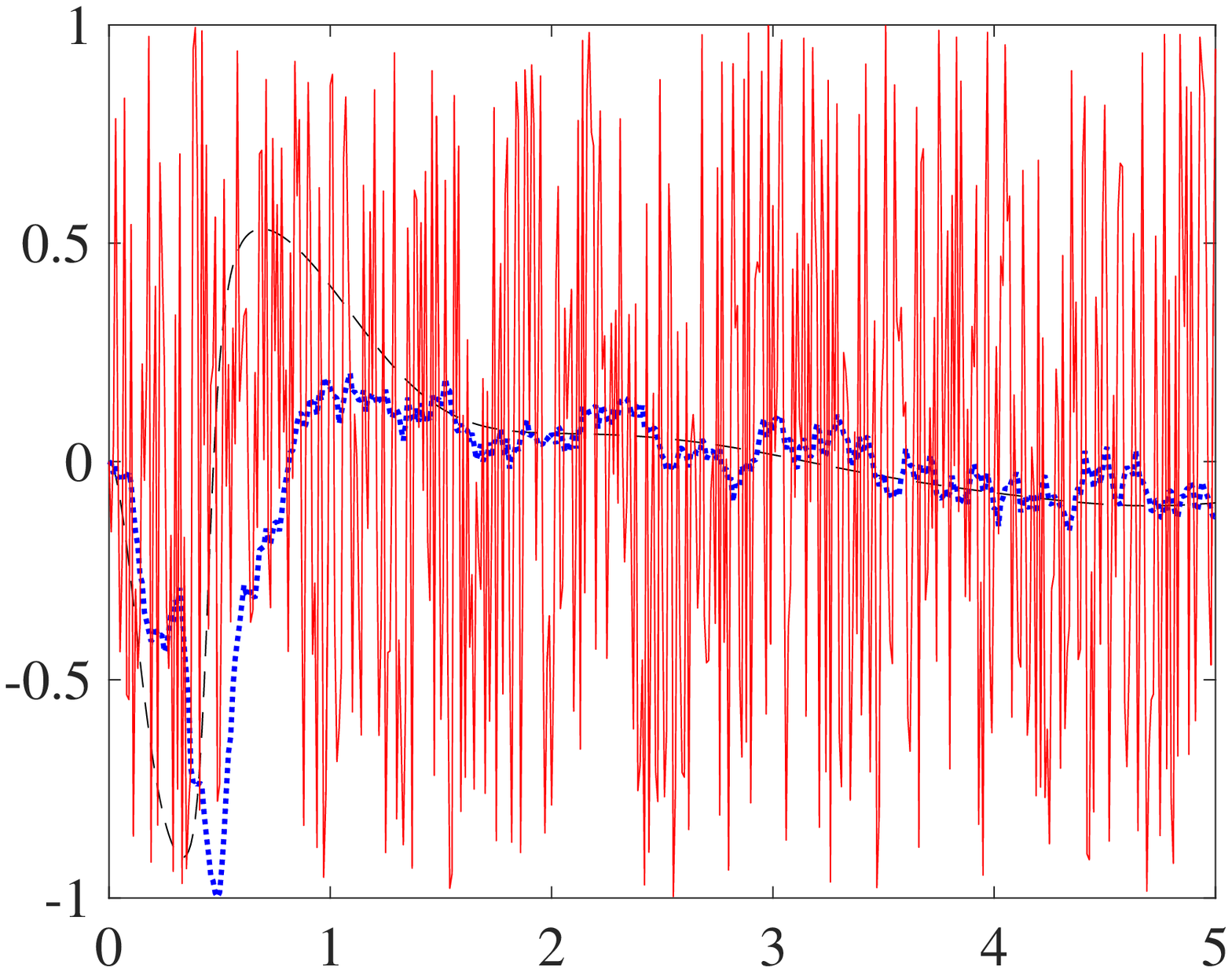}\label{fig:sim_R11_4}}}\centerline{
	        \subfigure[$R_{21}$]{
		\includegraphics[width=0.33\columnwidth]{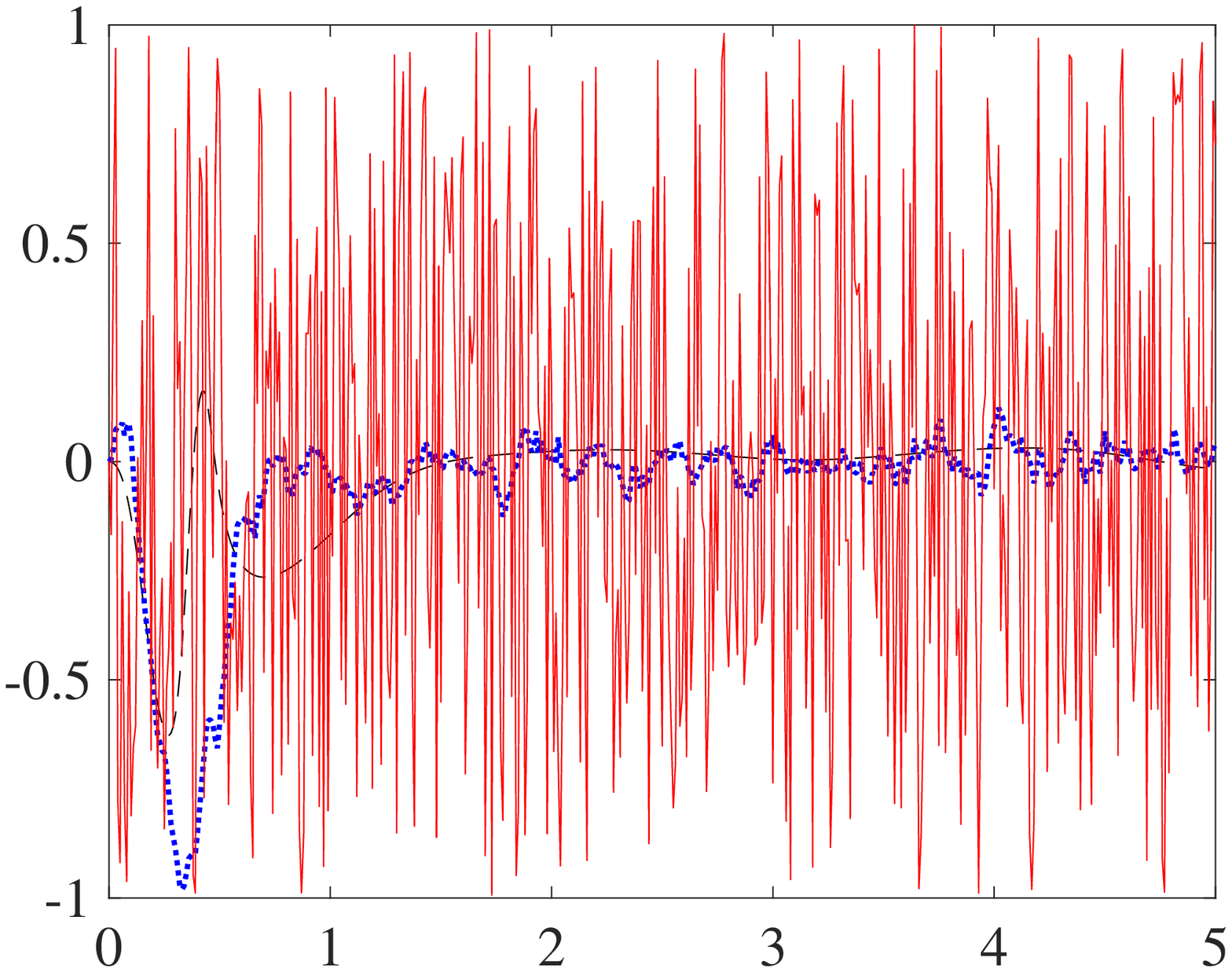}\label{fig:sim_R11_4}}\hspace{-0.3cm}
		\subfigure[$R_{22}$]{
		\includegraphics[width=0.33\columnwidth]{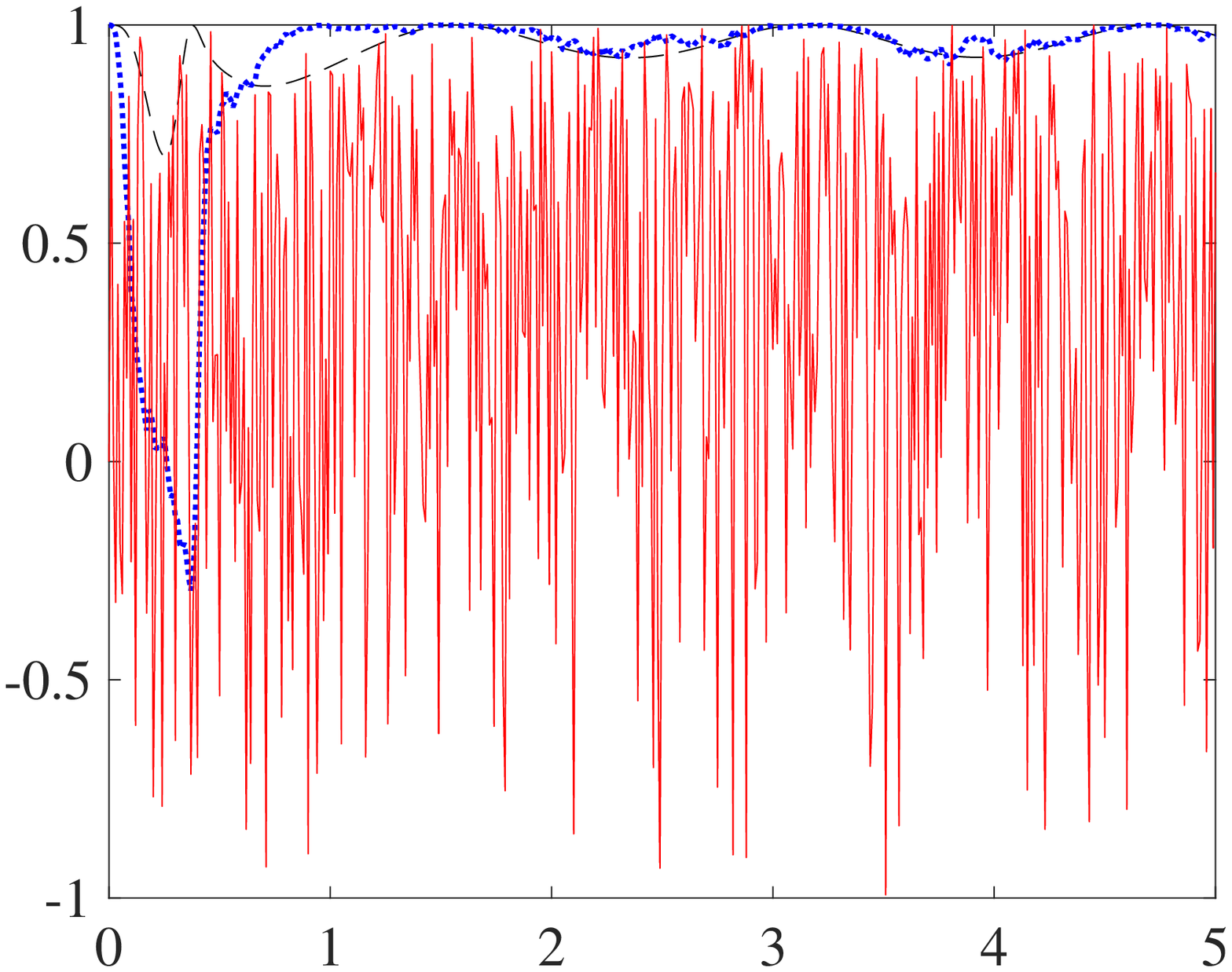}\label{fig:sim_R11_4}}\hspace{-0.3cm}
		\subfigure[$R_{23}$]{
		\includegraphics[width=0.33\columnwidth]{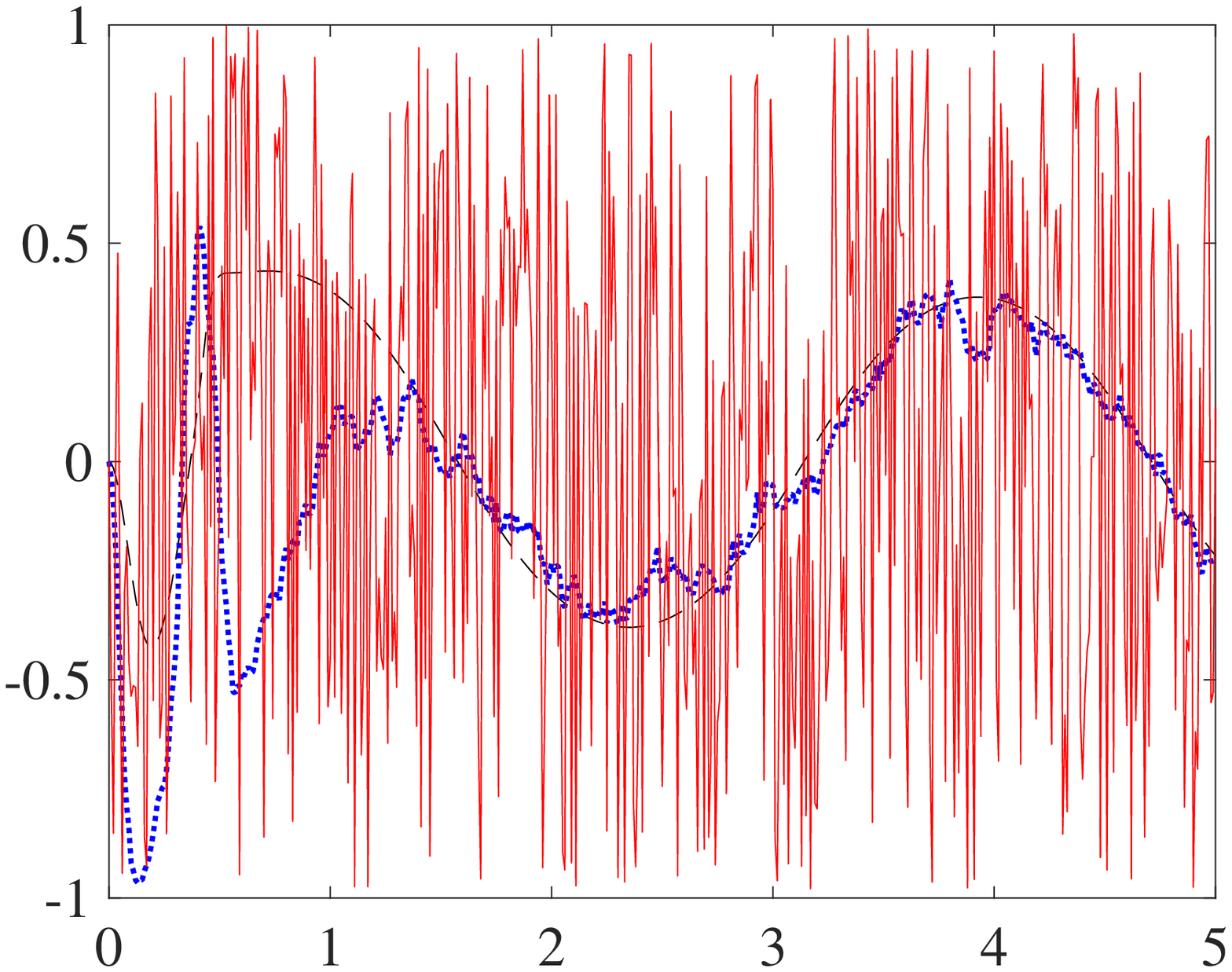}\label{fig:sim_R11_4}}}\centerline{
		\subfigure[$R_{31}$]{
		\includegraphics[width=0.33\columnwidth]{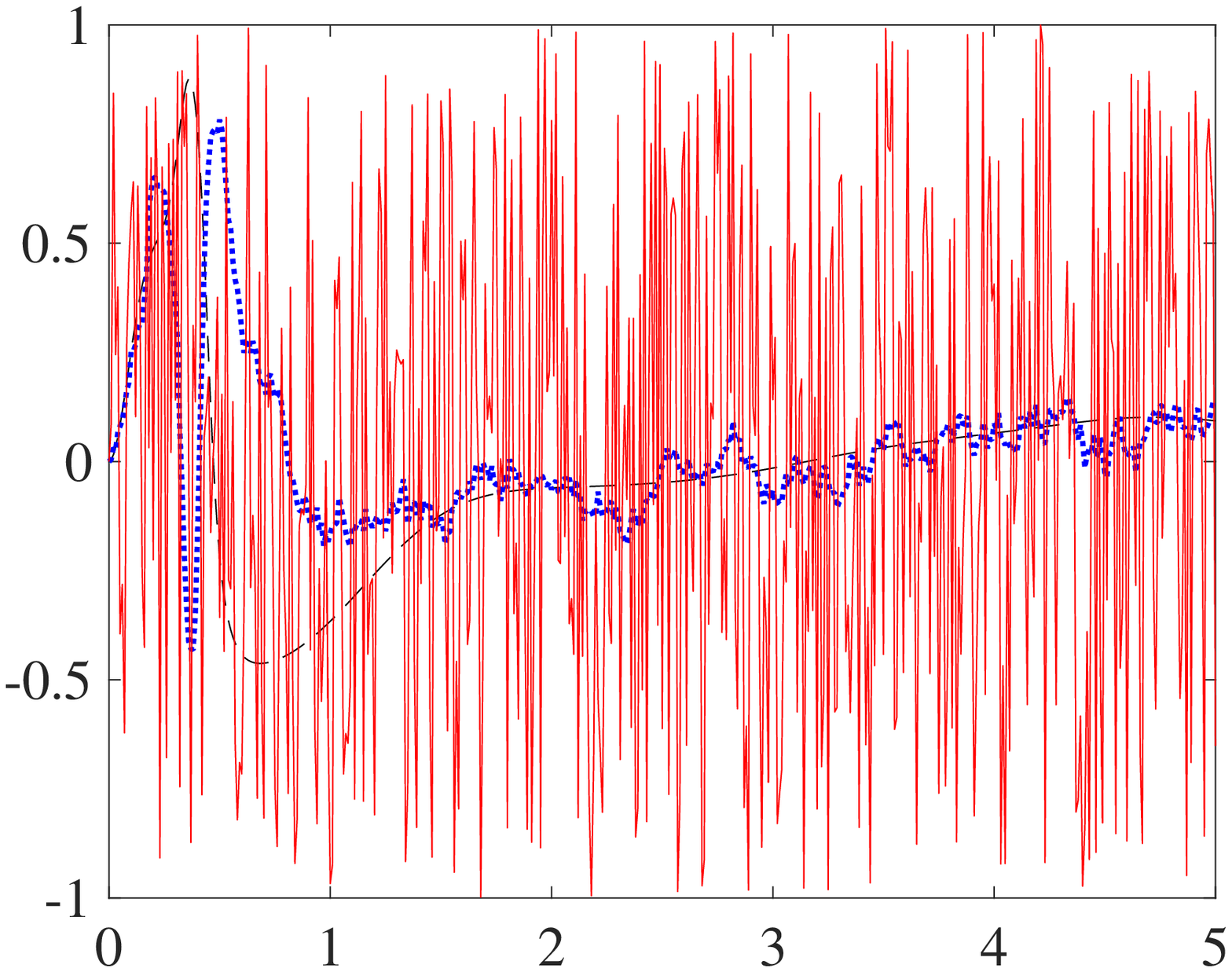}\label{fig:sim_R11_4}}		\hspace{-0.3cm}	
		 \subfigure[$R_{32}$]{
		\includegraphics[width=0.33\columnwidth]{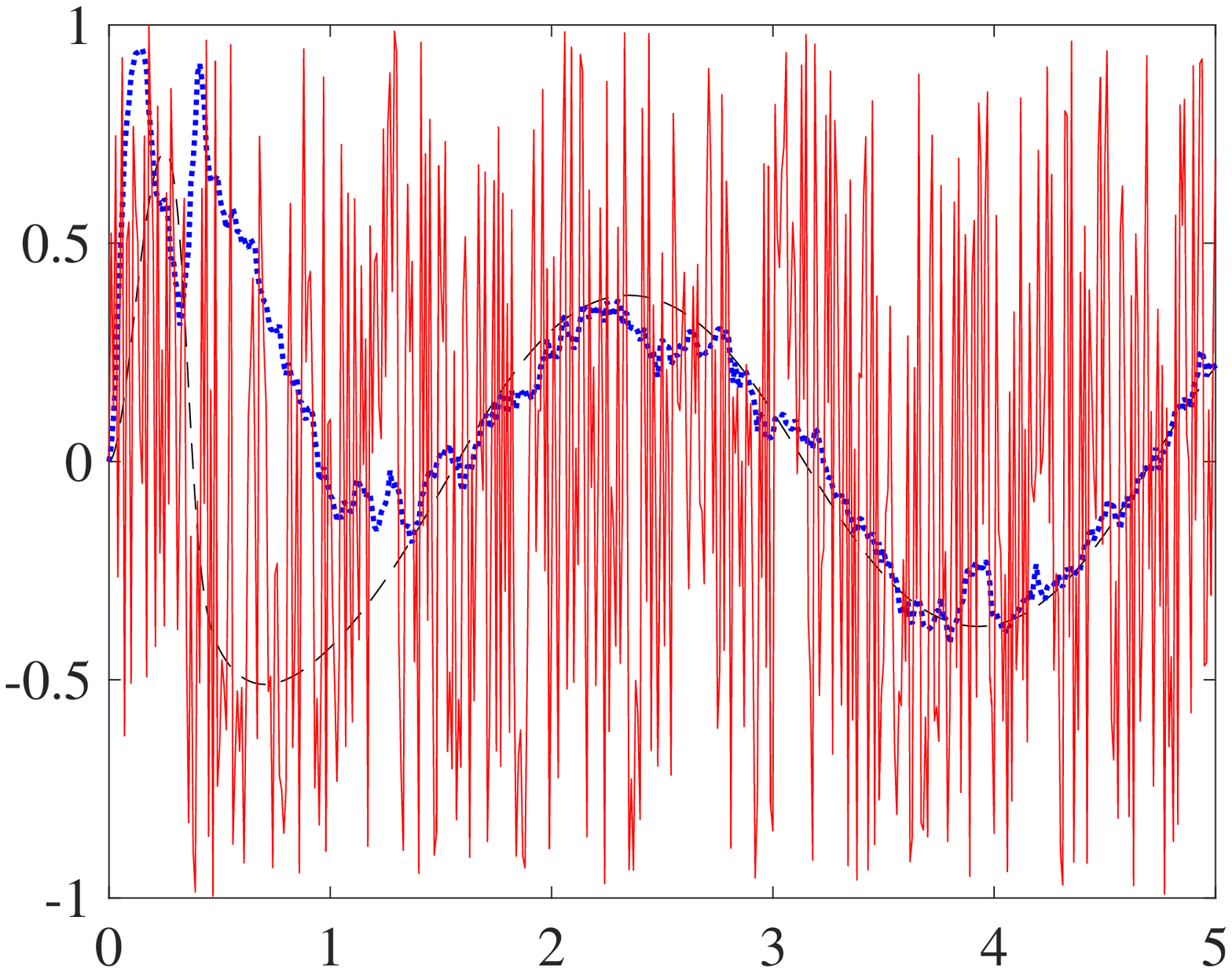}\label{fig:sim_R11_4}}\hspace{-0.3cm}
		\subfigure[$R_{33}$]{
		\includegraphics[width=0.33\columnwidth]{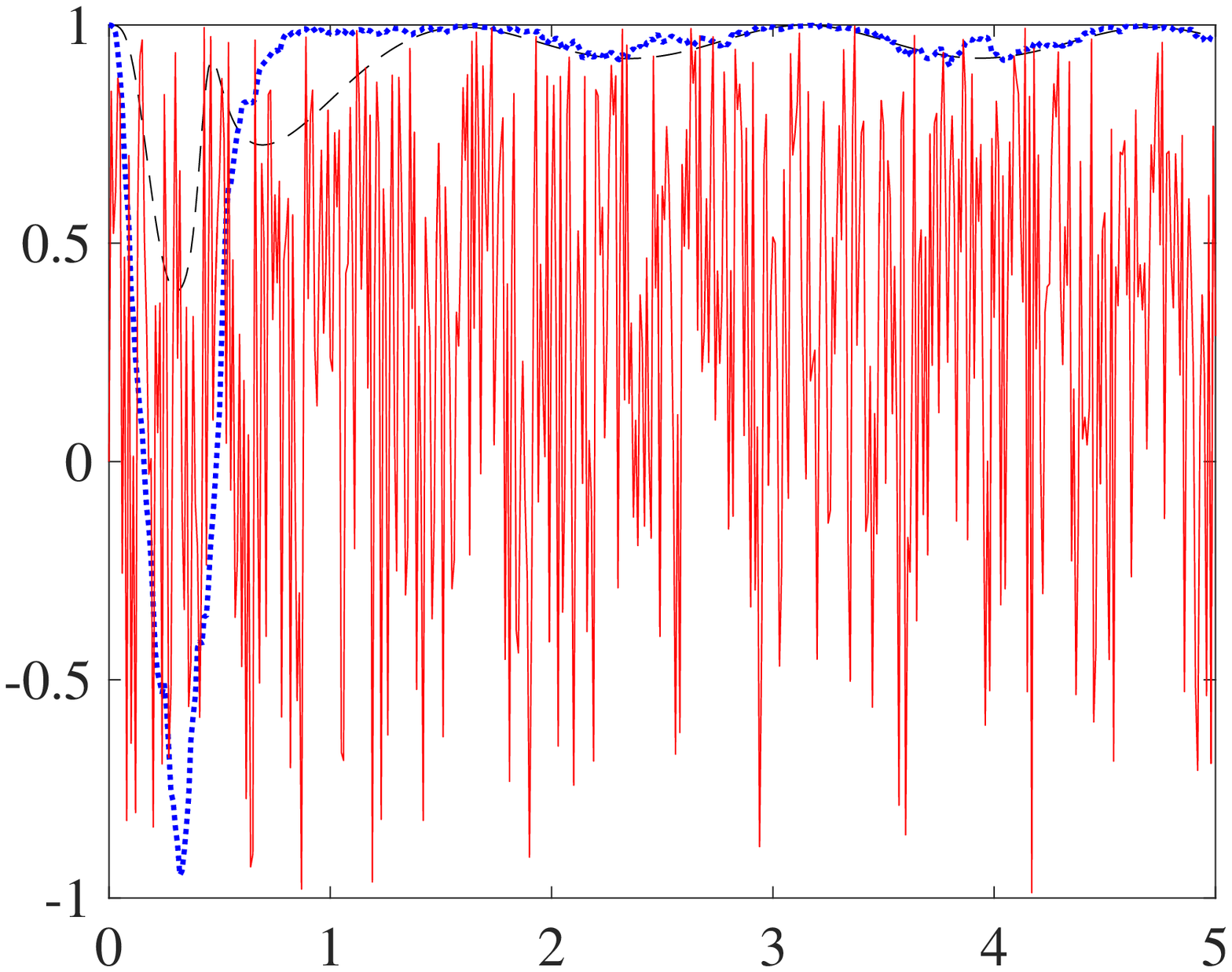}\label{fig:sim_R11_4}}
}
\caption{EKF performance in the first numerical example for quadrotor rotation matrix (dotted: EKF, dashed: desired, solid: noisy measurements). A short animation is also available at  \href{https://youtu.be/Dfi3IljfS-U}{https://youtu.be/Dfi3IljfS-U}}\label{fig:sim1R}
\end{figure}
We have also illustrated the attitude of the quadrotor as a rotation matrix in Figure \ref{fig:sim1R}. Noisy measurement data from sensor are presented along with the estimated attitude obtained from the proposed EKF and the desired attitude, $R_{d}$.

\subsection{\textbf{Example 2- GPS denied environment}}
Assume that we receive measurements for attitude and angular velocity from an (Inertial Measurement Unit), IMU installed onboard and there is no sensor or measurements available for position and translational velocity for the quadrotor. In other words, we present a scenario where GPS systems fails to work. In this case our measurement vector, $z\in\Re^3\times\SO$ consists of the angular velocity and the gravitational acceleration. The measurements noise covariance is chosen as $\mathcal{R}=0.1$ with process noise of $\mathcal{Q}=0.001$ and initial estimates of the state variables are given by $\bar{x}(0)=[0.2,-0.5,-0.5]^{T}$ and $\bar{{v}}(0)=[0.1,-0.1,-0.1]^{T}$. Desired trajectory is chosen to be an Elliptic Helix as
\begin{gather*}
x_d(t) =[0.4t,\; a\sin(wt),\; -b\cos(wt)]^{T},\nonumber\\
b_{1d}=[\cos(wt),\; \sin(wt),\; 0]^{T},
\end{gather*}
where constants $a=0.4$, $b=0.6$, $w=\pi$ and chosen particularly with several rotations of the quadrotor to illustrate the effectiveness of the proposed controller and extended kalman filter on $\SE$ which does not concerns with singularities and ambiguities as seen in derivations with Euler angles or quaternions. 
\begin{figure}[h]
\centerline{
	        \subfigure[Quadrotor position, $x$, $\hat{x}$]{
		\includegraphics[width=0.5\columnwidth]{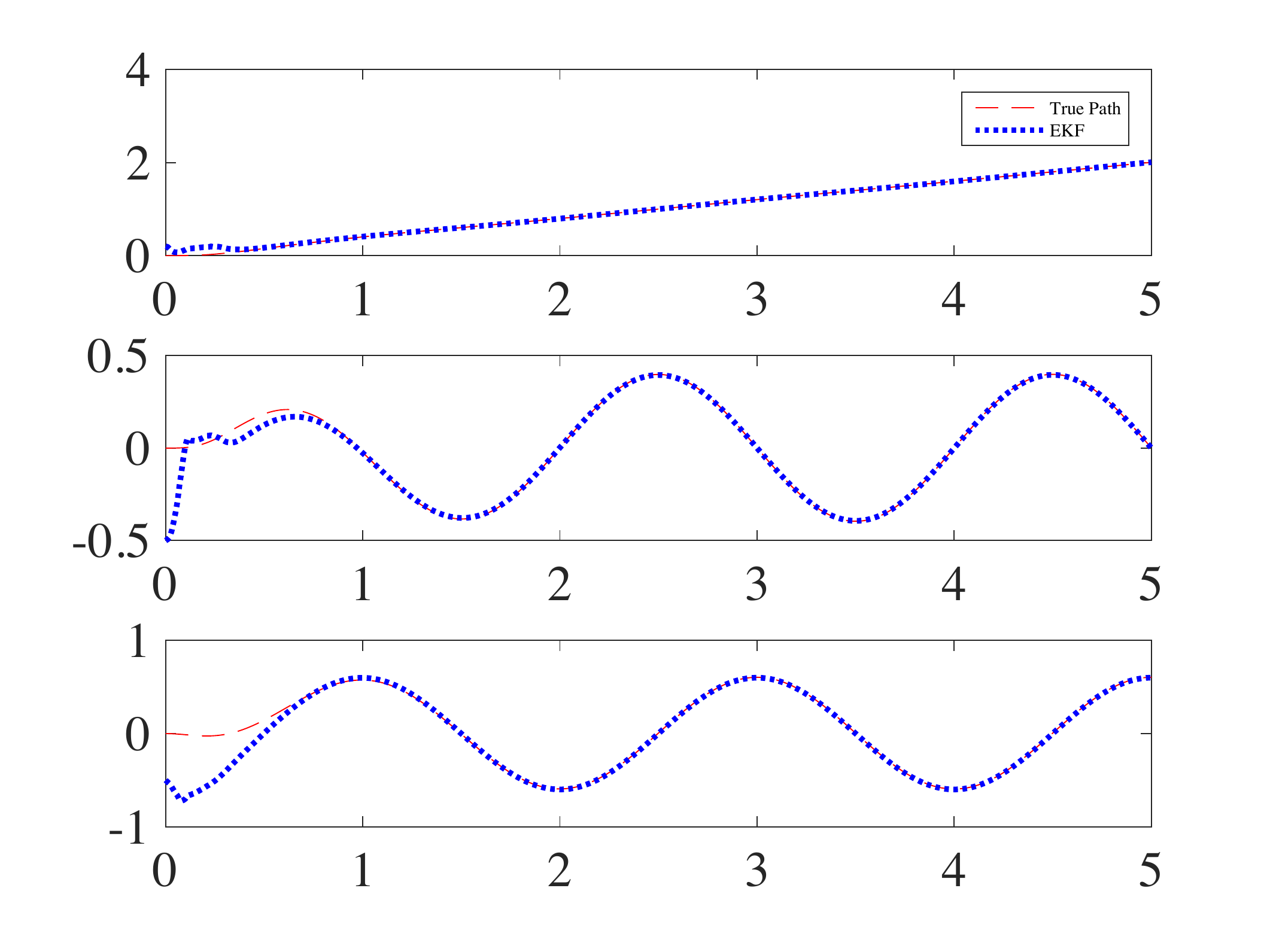}\label{fig:sim_x}}
		\subfigure[$\Omega_{1}$, $\hat{\Omega}_{1}$]{
		\includegraphics[width=0.5\columnwidth]{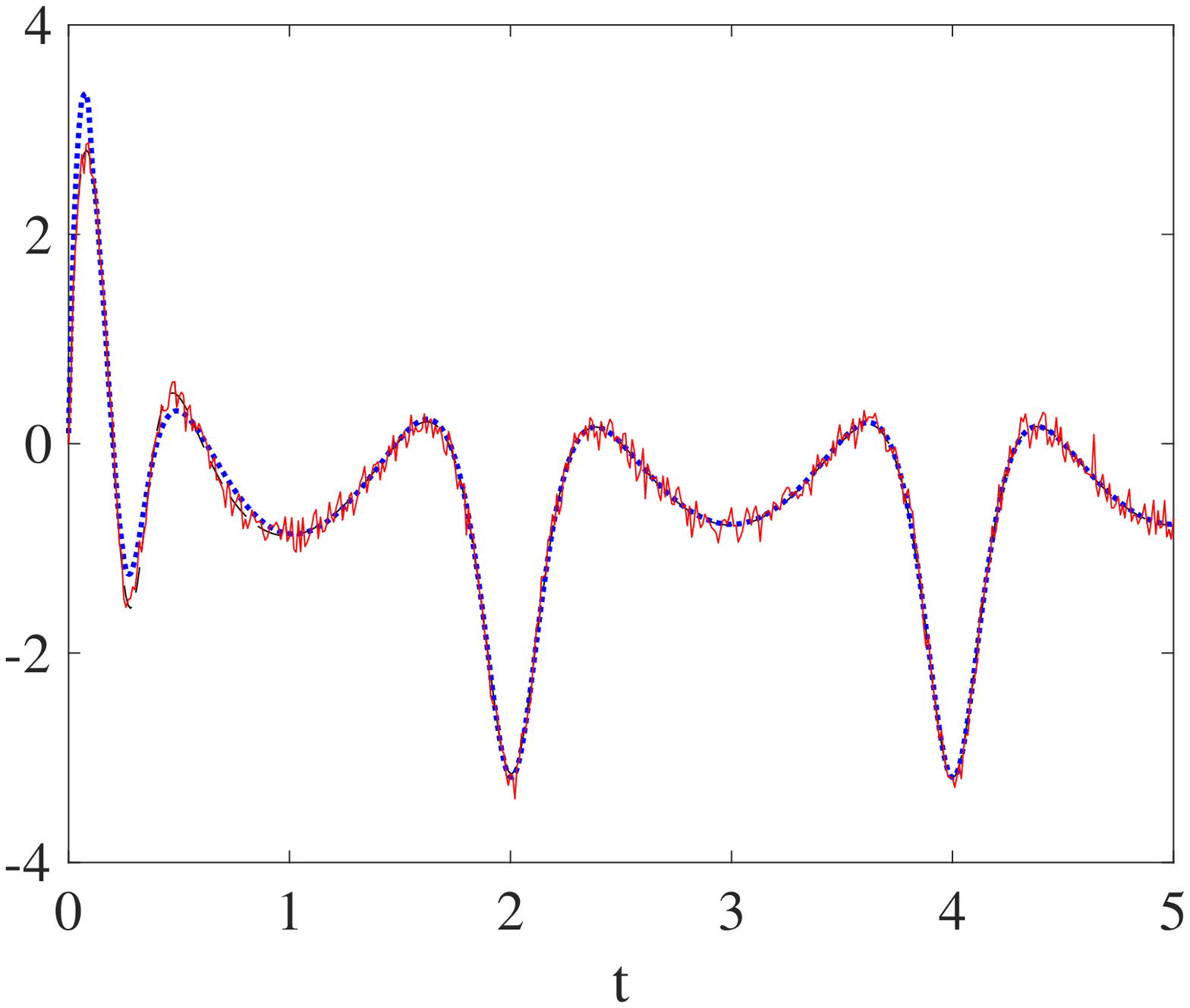}\label{fig:sim_W1}}}\centerline{
	        \subfigure[Quadrotor velocity, $v$, $\hat{v}$]{
		\includegraphics[width=0.5\columnwidth]{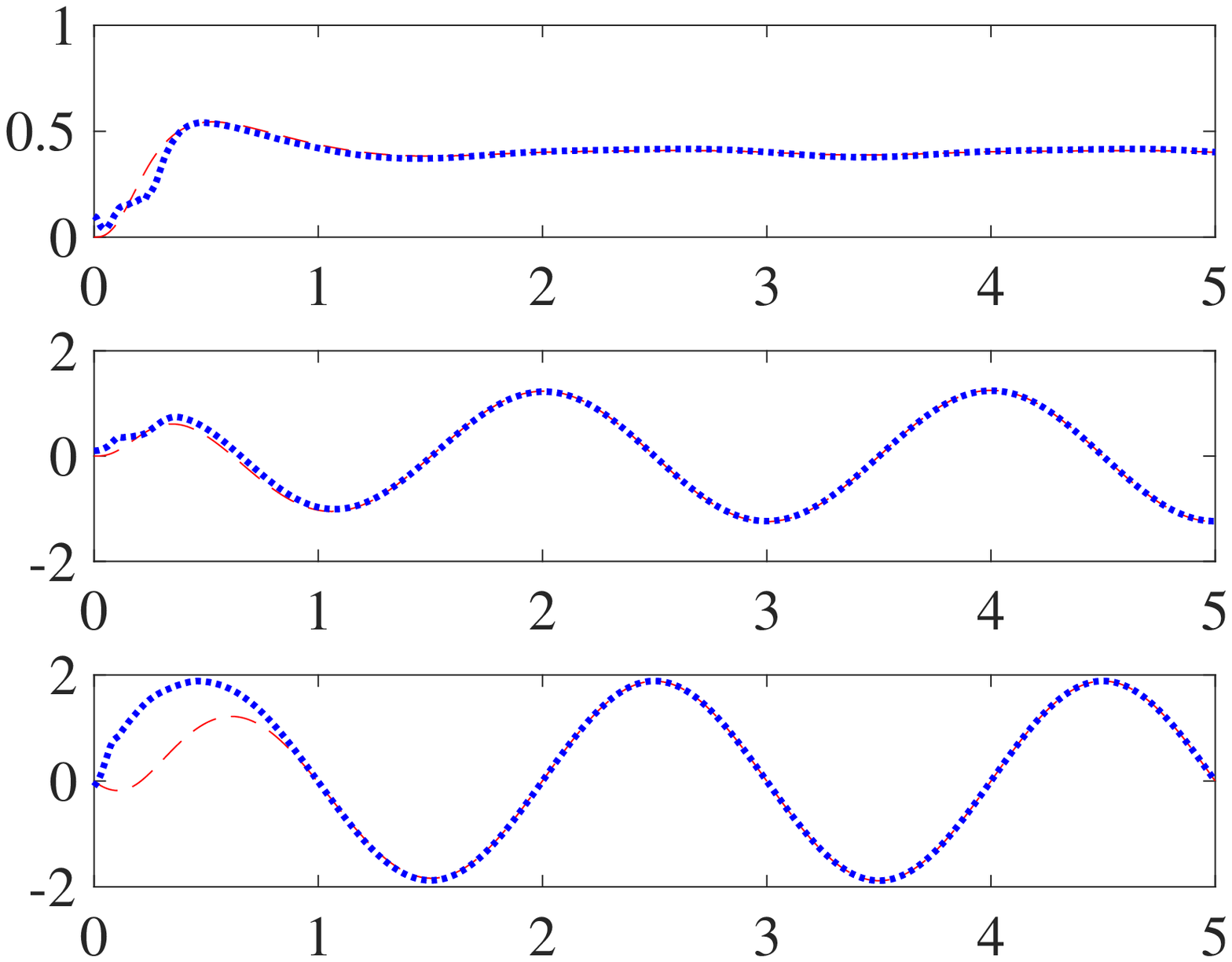}\label{fig:sim_v}}
		\subfigure[$\Omega_{2}$, $\hat{\Omega}_{2}$]{
		\includegraphics[width=0.5\columnwidth]{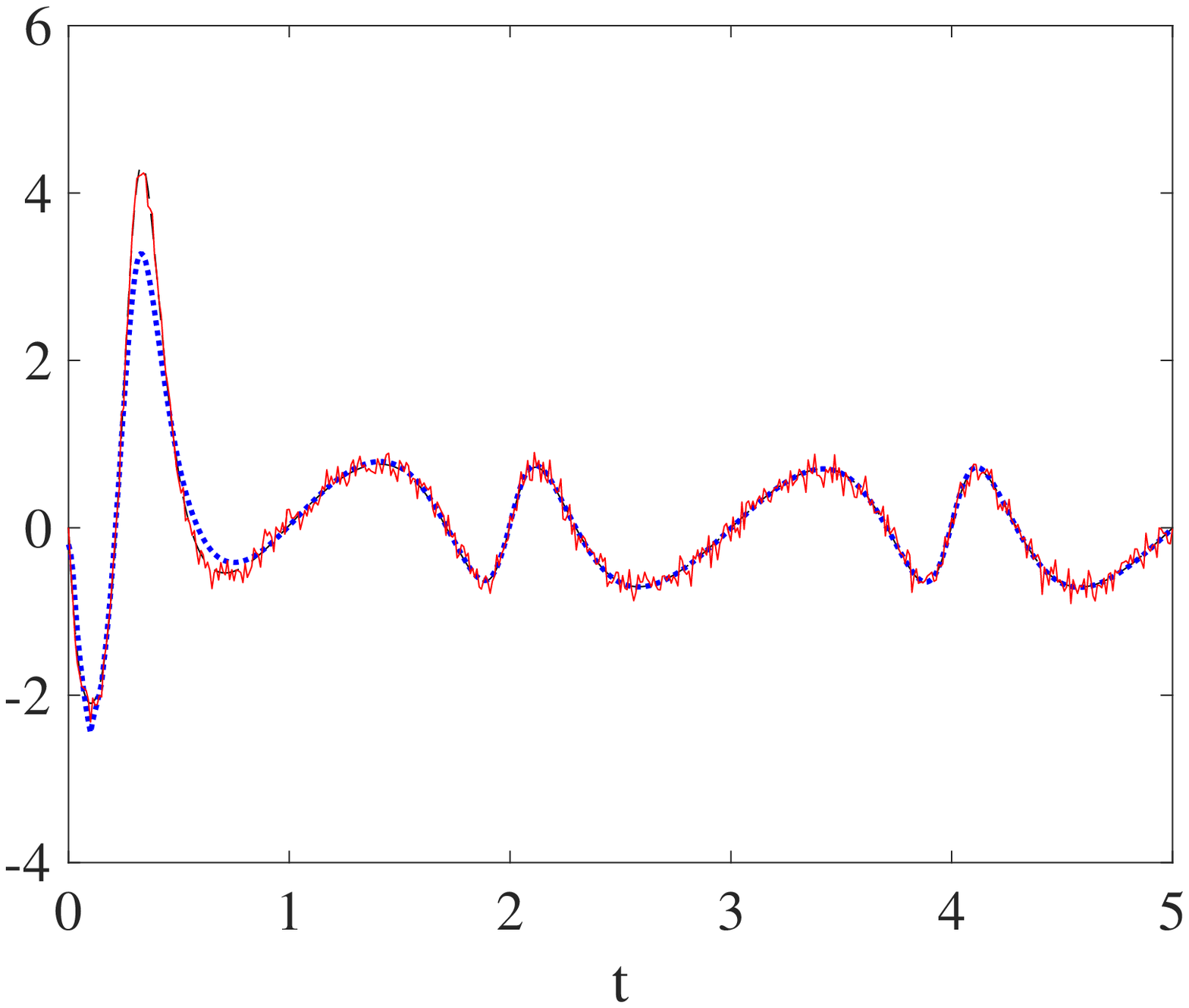}\label{fig:sim_W2}}}\centerline{
		\subfigure[3D view, $x$, $\hat{x}$]{
		\includegraphics[width=0.5\columnwidth]{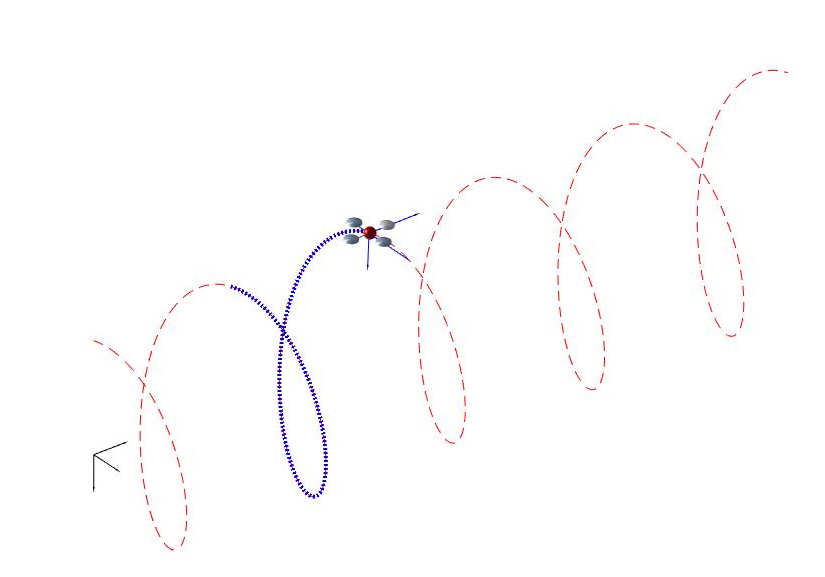}\label{fig:eq}}			
		 \subfigure[$\Omega_{3}$, $\hat{\Omega}_{3}$]{
		\includegraphics[width=0.5\columnwidth]{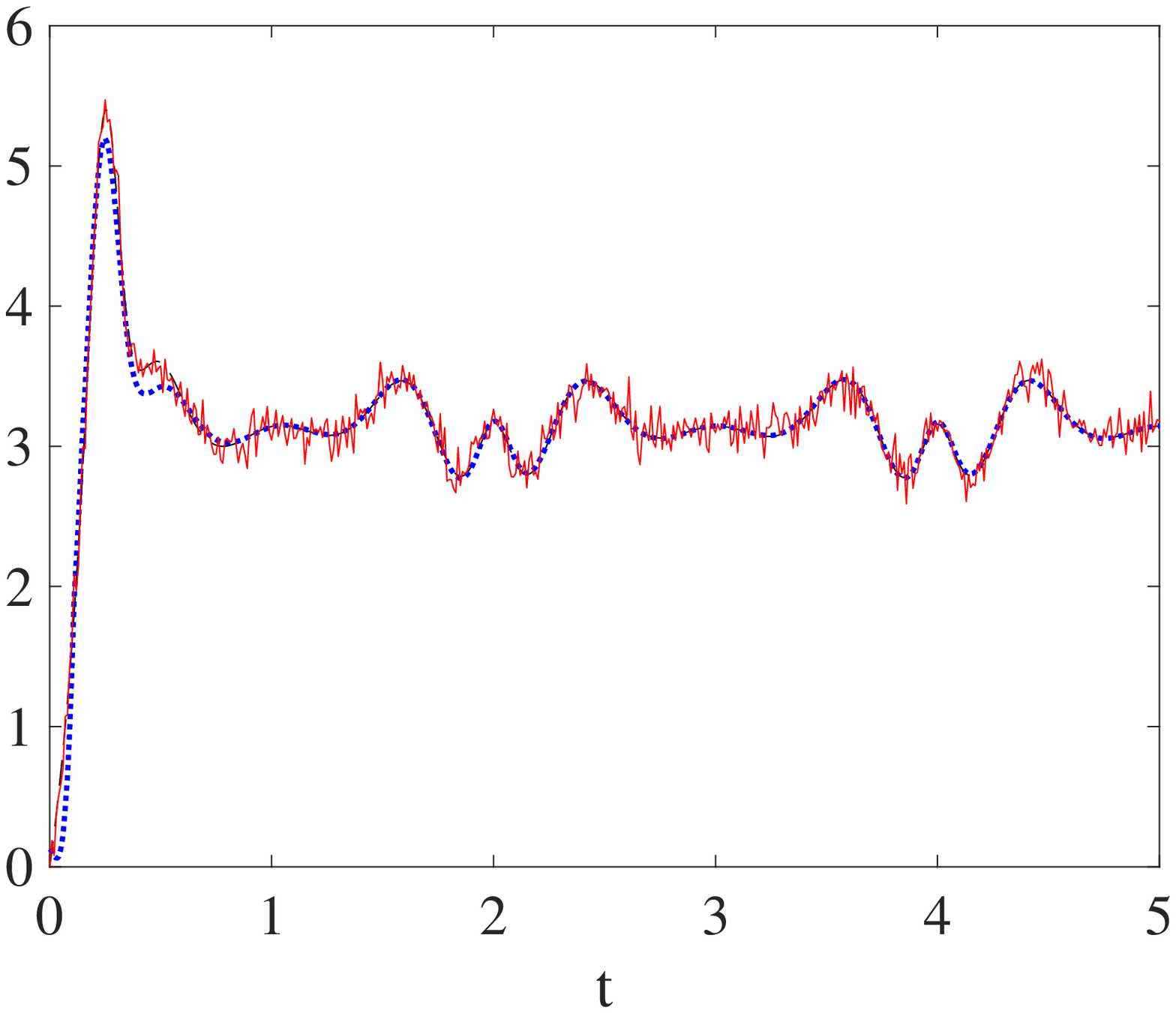}\label{fig:sim_W3}}}\centerline{
		\subfigure[Position error]{
		\includegraphics[width=0.5\columnwidth]{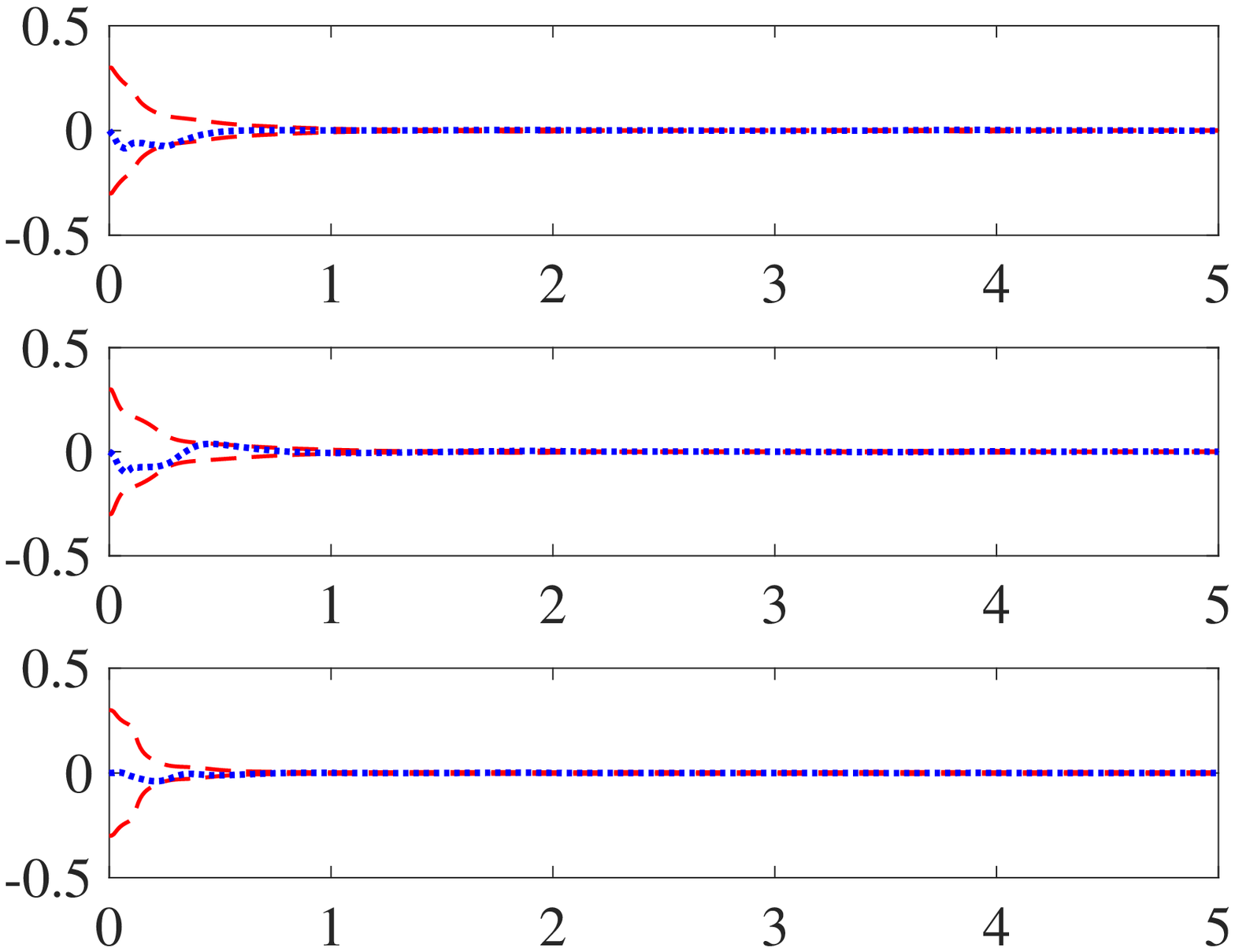}\label{fig:errorsx}}
		\subfigure[Velocity error]{
		\includegraphics[width=0.5\columnwidth]{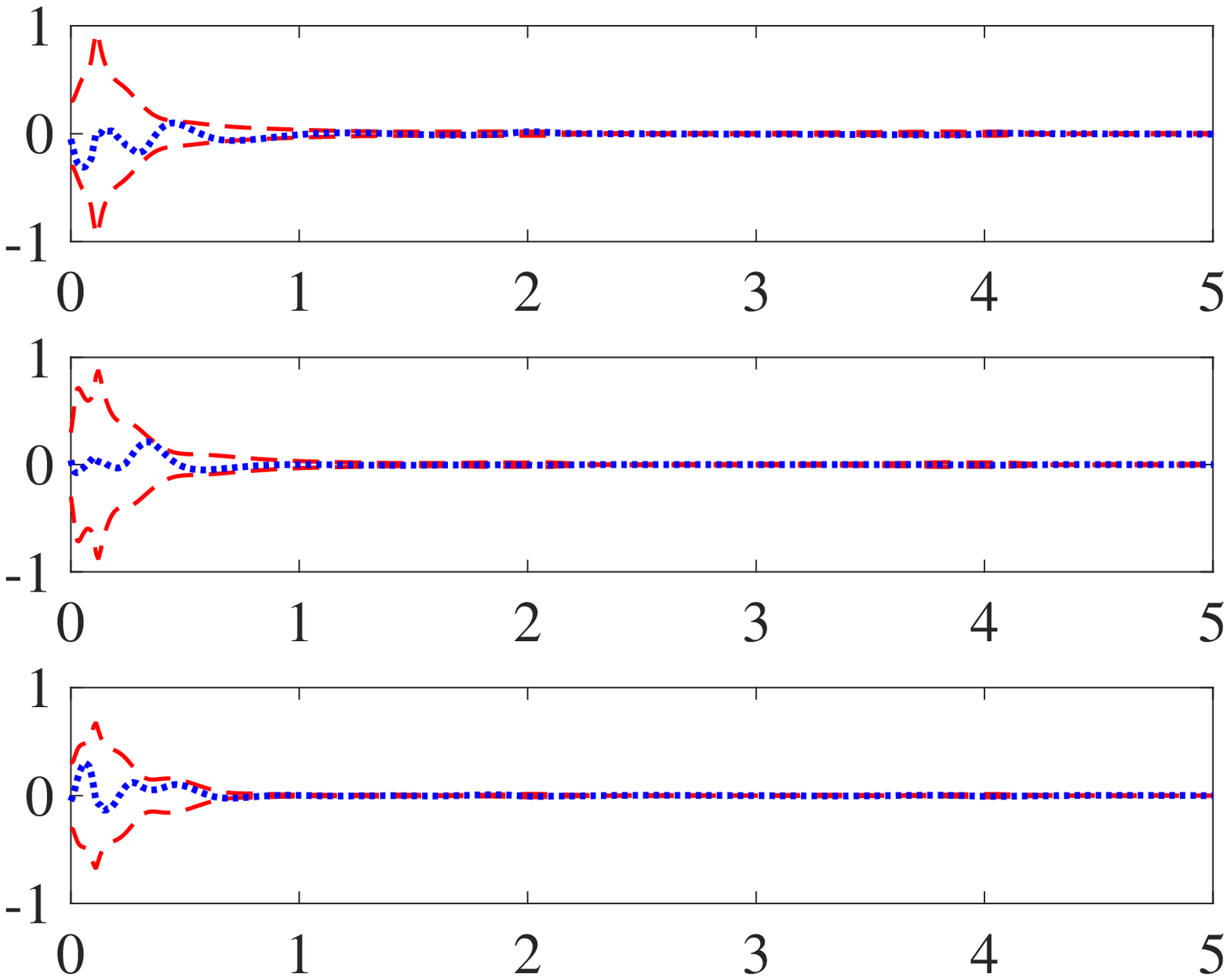}\label{fig:errorsv}}
}
\caption{EKF performance in the second numerical example (dotted: EKF, dashed: desired, solid: noisy measurements). A short animation is also available at  \href{https://youtu.be/F4Vntws97RU}{https://youtu.be/F4Vntws97RU}}\label{fig:sim2}
\end{figure}
Figures \ref{fig:sim_x} and \ref{fig:sim_v} illustrate the position and velocity of the quadrotor estimated with EKF and plotted verse the true path respectively. As it is clear from this figure, although we did not have measurements on position and translational velocity, EKF was able to estimate this values significantly. Position and velocity estimation errors are also presented in Figures \ref{fig:errorsx} and \ref{fig:errorsv} respectively. \\
Figures \ref{fig:sim_W1}, \ref{fig:sim_W2}, and \ref{fig:sim_W3} illustrate the angular velocity of the quadrotor during this maneuver. Noisy measurement data and the estimated value from EKF are presented along with the true path to show the performance and effectiveness of the EKF to reduce the error and improve noisy measurements. Position and velocity estimation errors for this example are presented in Figures \ref{fig:errorsx} and \ref{fig:errorsv}.



\section{Experimental Results}\label{sec:EXP}
The quadrotor UAV developed at the flight dynamics and control laboratory at the George Washington University is shown at Figure~\ref{fig:Quad}. We developed an accurate CAD model as shown in Figure~\ref{fig:QM} to identify several parameters of the quadrotor, such as moment of inertia and center of mass. Furthermore, a precise rotor calibration is performed for each rotor, with a custom-made thrust stand as shown in Figure~\ref{fig:stand} to determine the relation between the command in the motor speed controller and the actual thrust. For various values of motor speed commands, the corresponding thrust is measured, and those data are fitted with a second order polynomial. 
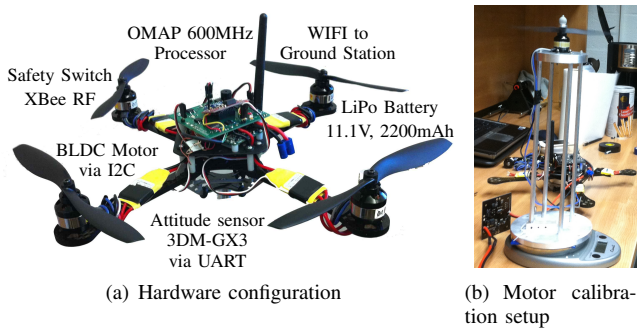
\begin{figure}[h]
\centerline{
	\subfigure[Hardware configuration]{
\setlength{\unitlength}{0.1\columnwidth}\scriptsize
\begin{picture}(7,4)(0,0)
\put(0,0){\includegraphics[width=0.7\columnwidth]{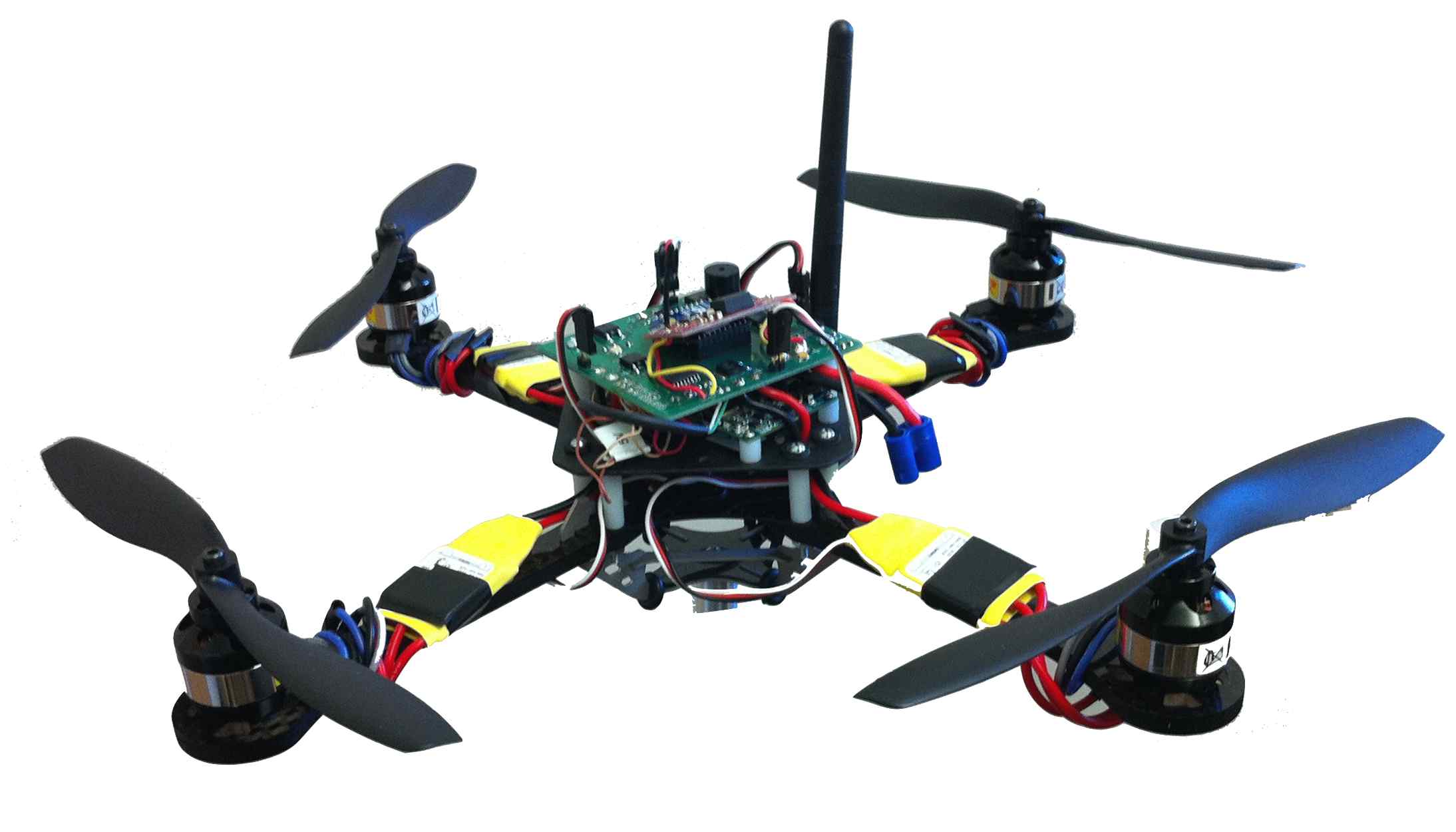}}
\put(1.95,3.2){\shortstack[c]{OMAP 600MHz\\Processor}}
\put(2.3,0){\shortstack[c]{Attitude sensor\\3DM-GX3\\ via UART}}
\put(0.85,1.4){\shortstack[c]{BLDC Motor\\ via I2C}}
\put(0.1,2.5){\shortstack[c]{Safety Switch\\XBee RF}}
\put(4.3,3.2){\shortstack[c]{WIFI to\\Ground Station}}
\put(5,2.0){\shortstack[c]{LiPo Battery\\11.1V, 2200mAh}}
\end{picture}\label{fig:Quad}}
	\subfigure[Motor calibration setup]{
	\includegraphics[width=0.25\columnwidth]{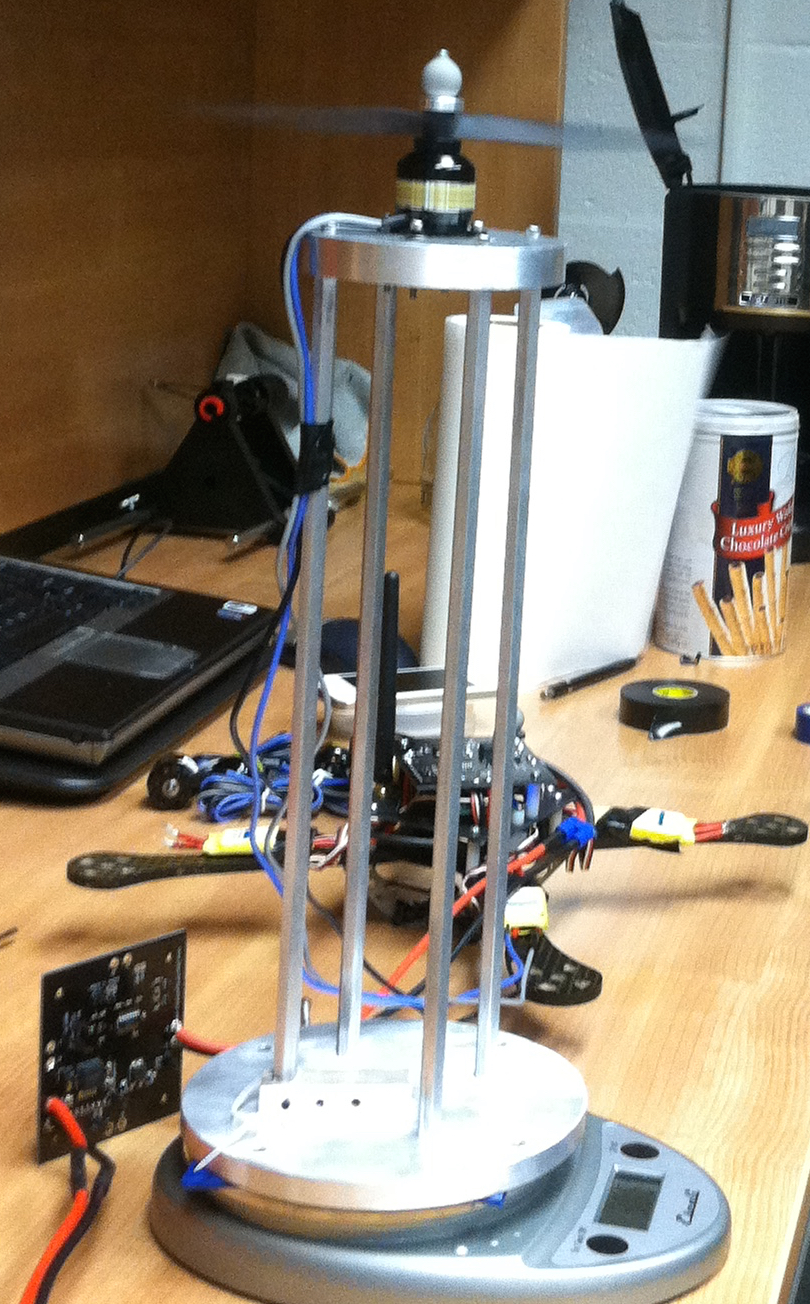}\label{fig:stand}}
}
\caption{Hardware development for a quadrotor UAV}
\end{figure}
Angular velocity and attitude are measured from inertial measurement unit (IMU). Position of the UAV is measured from motion capture system (Vicon). Ground computing system receives the Vicon data and send it to the UAV via XBee. The Gumstix is adopted as micro computing unit on the UAV. A multi-threaded C/C++ software developed for autonomous control of the quadrotor. It has two main threads, namely Vicon thread, control and estimation thread. The Vicon thread receives the Vicon measurement. In second thread, it receives the IMU measurement, estimates the velocity with EKF and handles the control outputs. Figures \ref{fig:exp} presents the experimental results for a maneuver where quadrotor autonomously tracks the following desired trajectory which is defined as a Lissajous curve
\begin{align*}
x_d (t) = [\sin(t)+\frac{\pi}{2},\;\sin 2(t),\; -0.3].
\end{align*} 
\begin{figure}[h]
\centerline{
	        \subfigure[$x$, $x_{d}$]{
		\includegraphics[width=0.5\columnwidth]{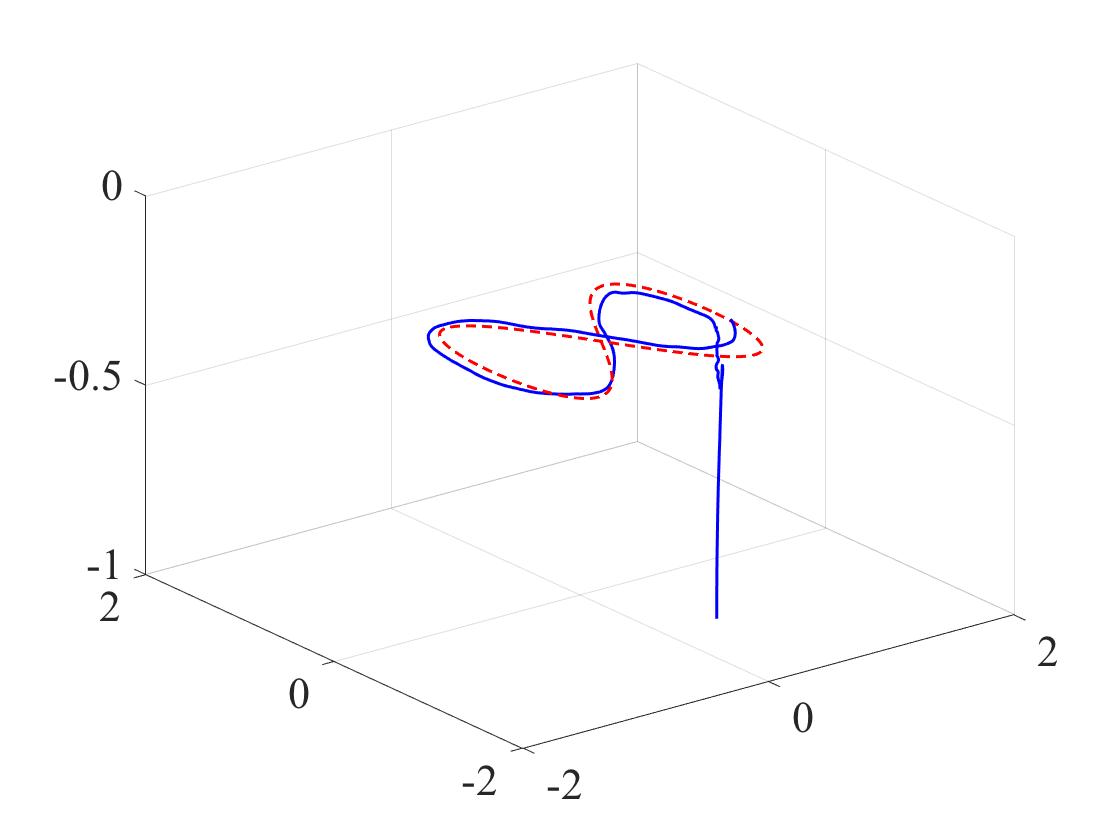}\label{fig:exp_x}}
		 \subfigure[$\Psi$, $e_{R}$, $e_{\Omega}$]{
		\includegraphics[width=0.5\columnwidth]{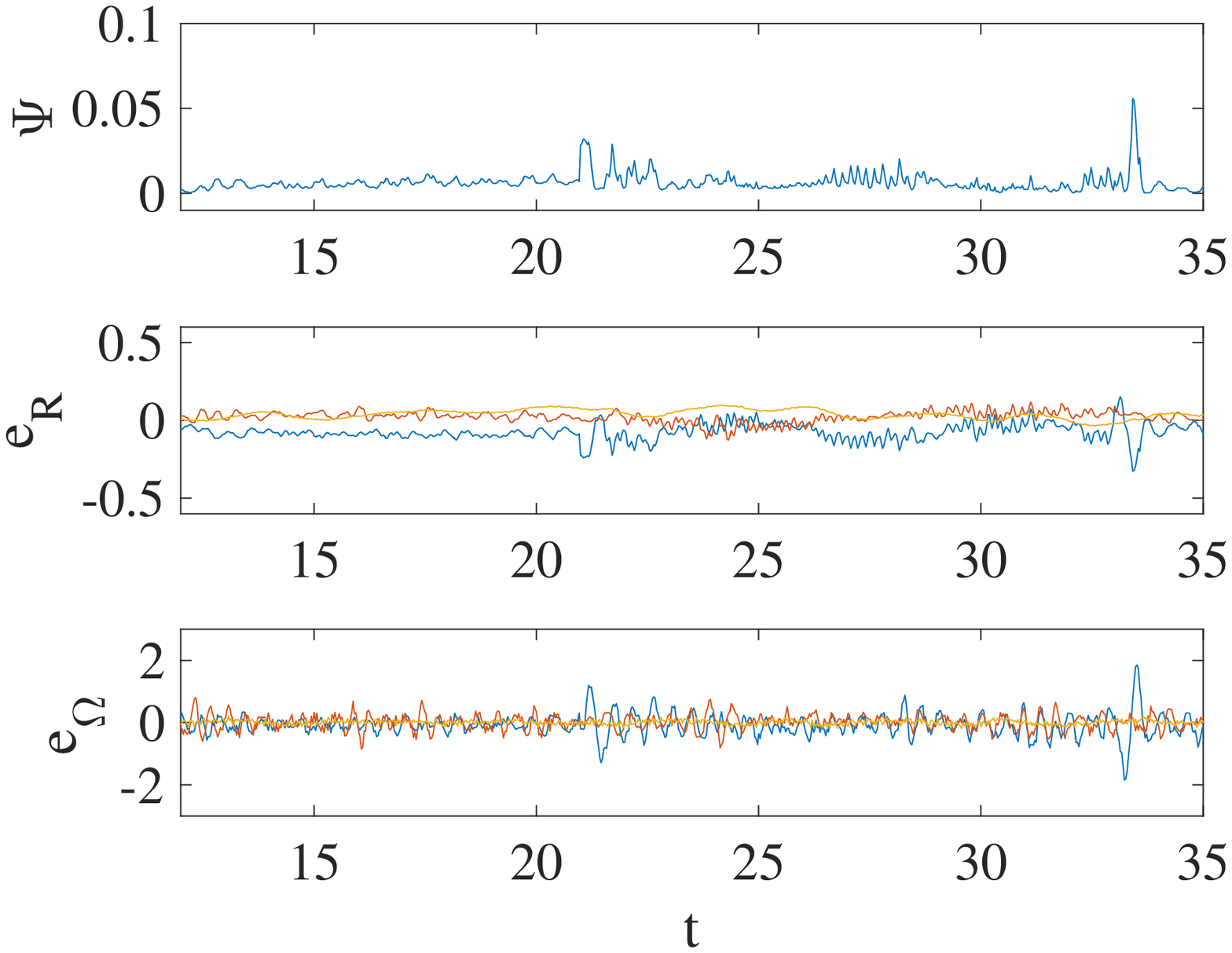}\label{fig:exp_error}}}\centerline{
		\subfigure[$\frac{dx_{1}}{dt}$, $v_{d_{1}}$, $\hat{v}_{1}$]{
		\includegraphics[width=0.5\columnwidth]{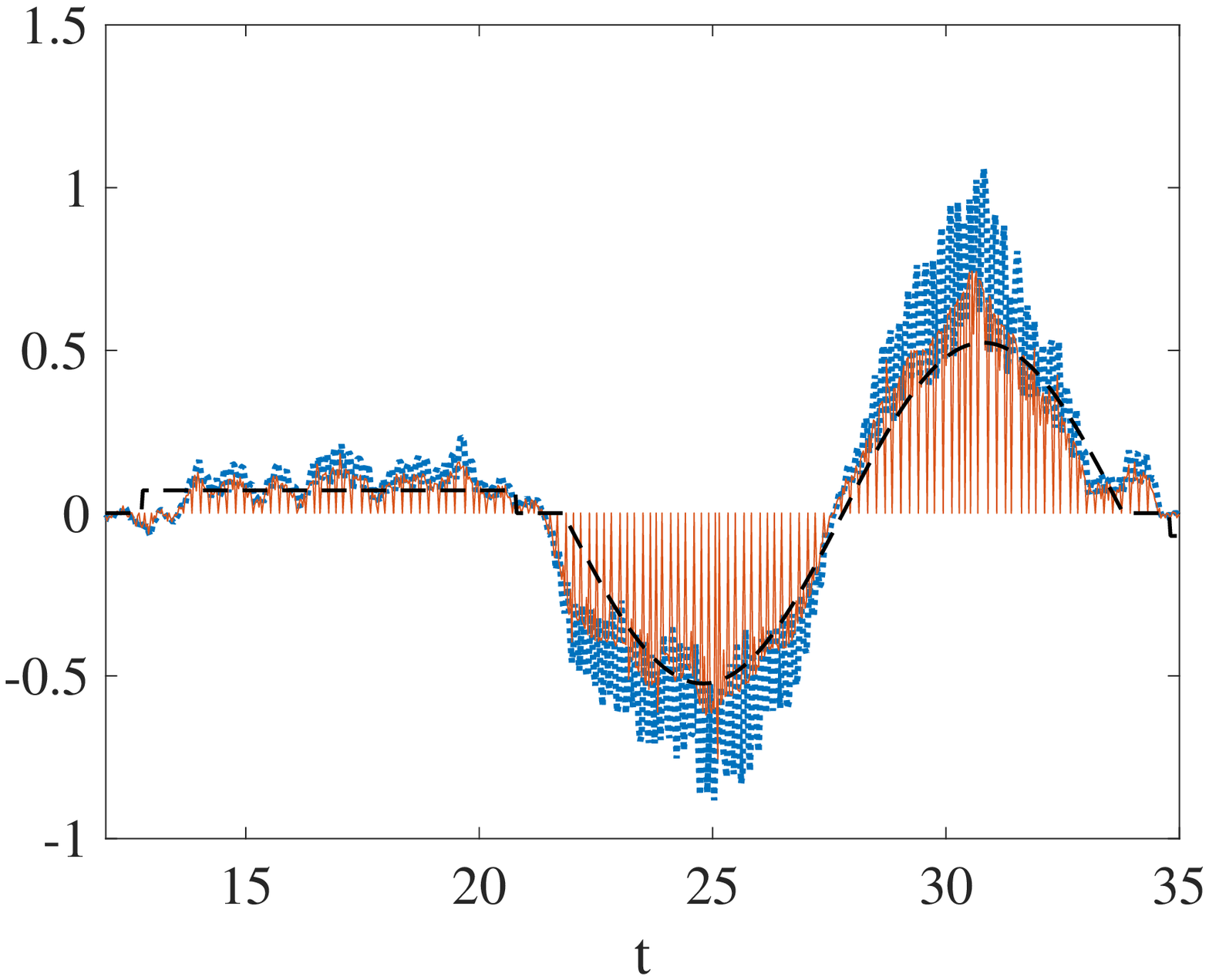}\label{fig:exp_v1}}
		 \subfigure[$\Omega_{1}$, $\hat{\Omega}_{1}$]{
		\includegraphics[width=0.5\columnwidth]{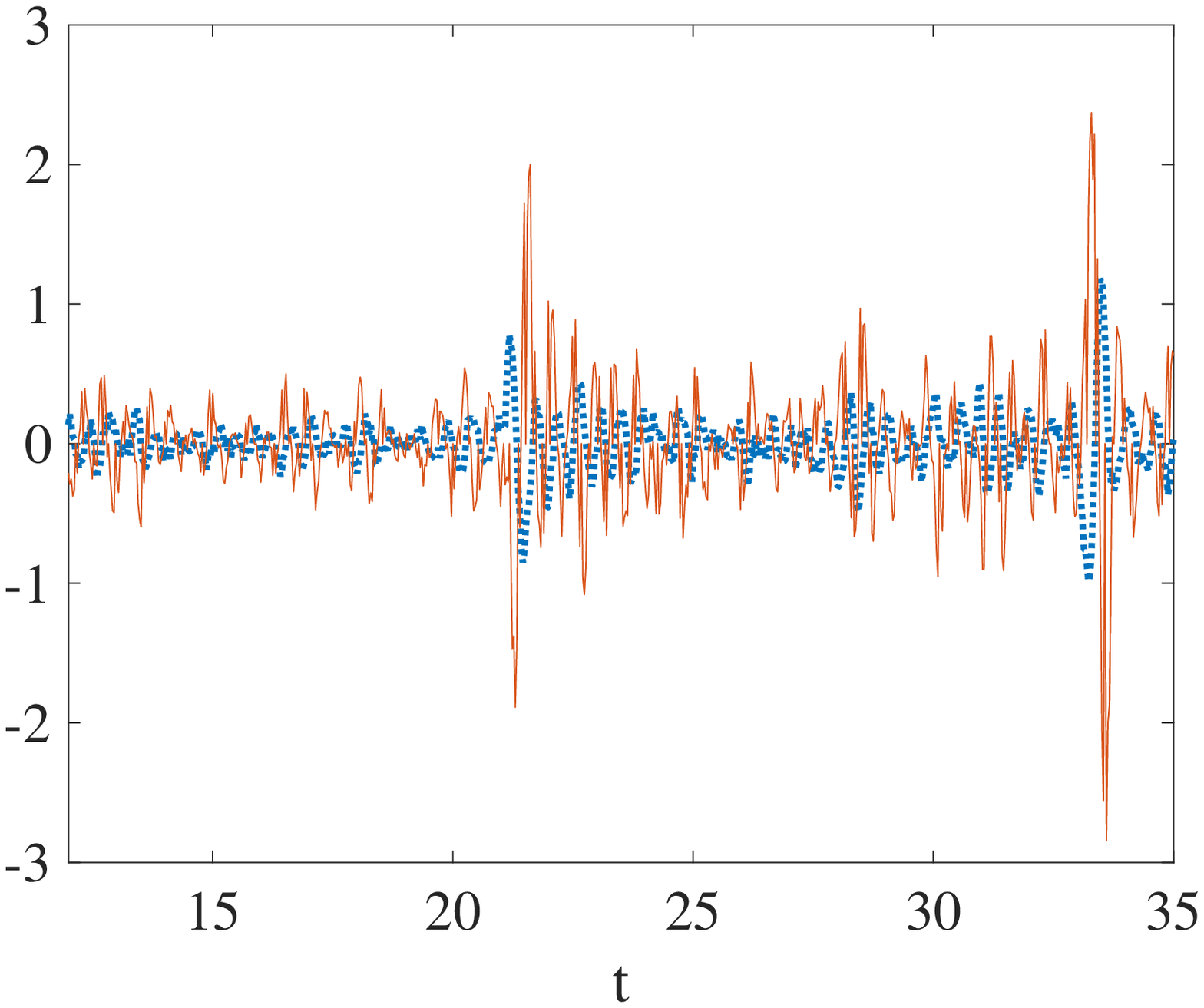}\label{fig:exp_W1}}}\centerline{
		 \subfigure[$\frac{dx_{2}}{dt}$, $v_{d_{1}}$, $\hat{v}_{2}$]{
		\includegraphics[width=0.5\columnwidth]{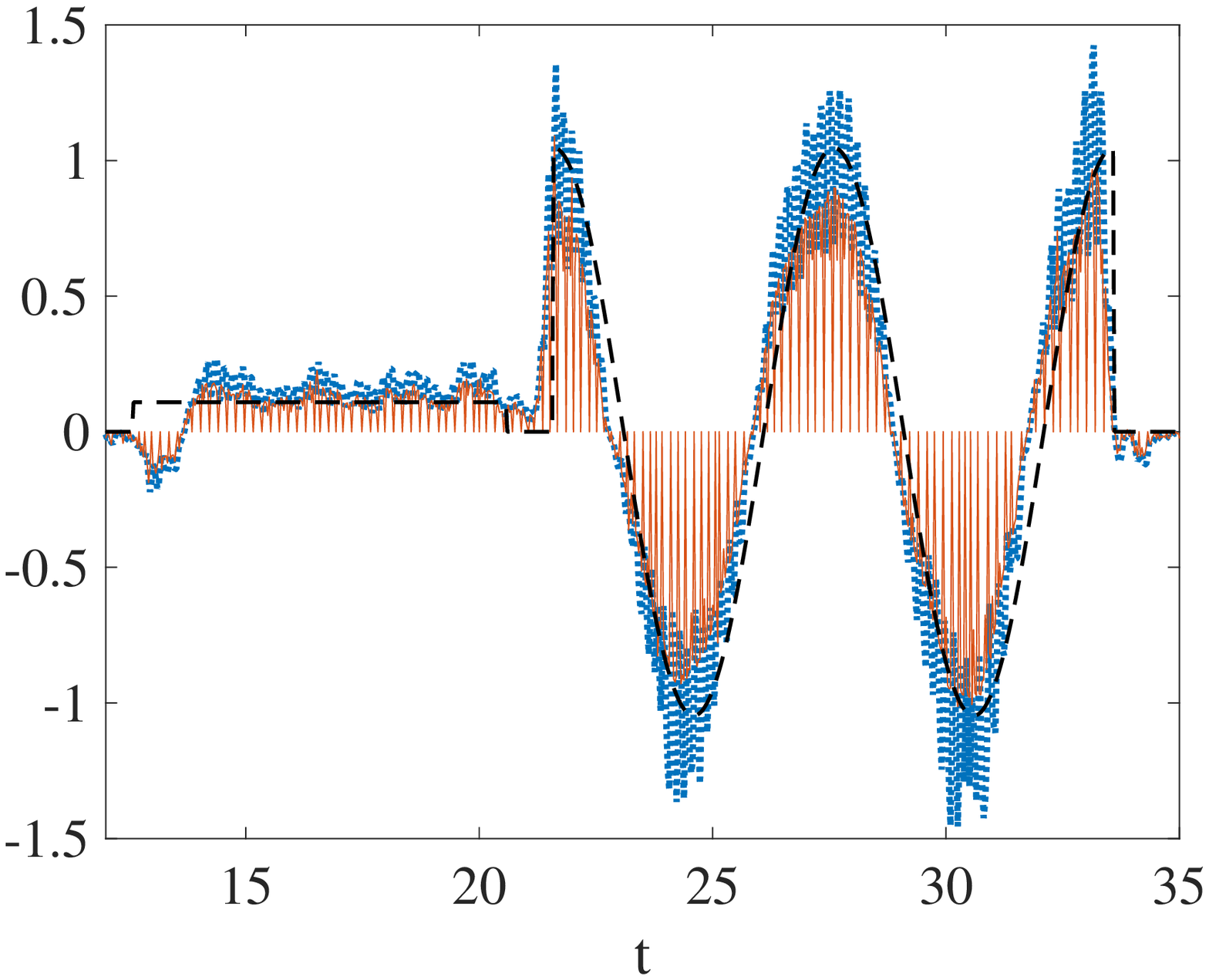}\label{fig:exp_v2}}
		 \subfigure[$\Omega_{2}$, $\hat{\Omega}_{2}$]{
		\includegraphics[width=0.5\columnwidth]{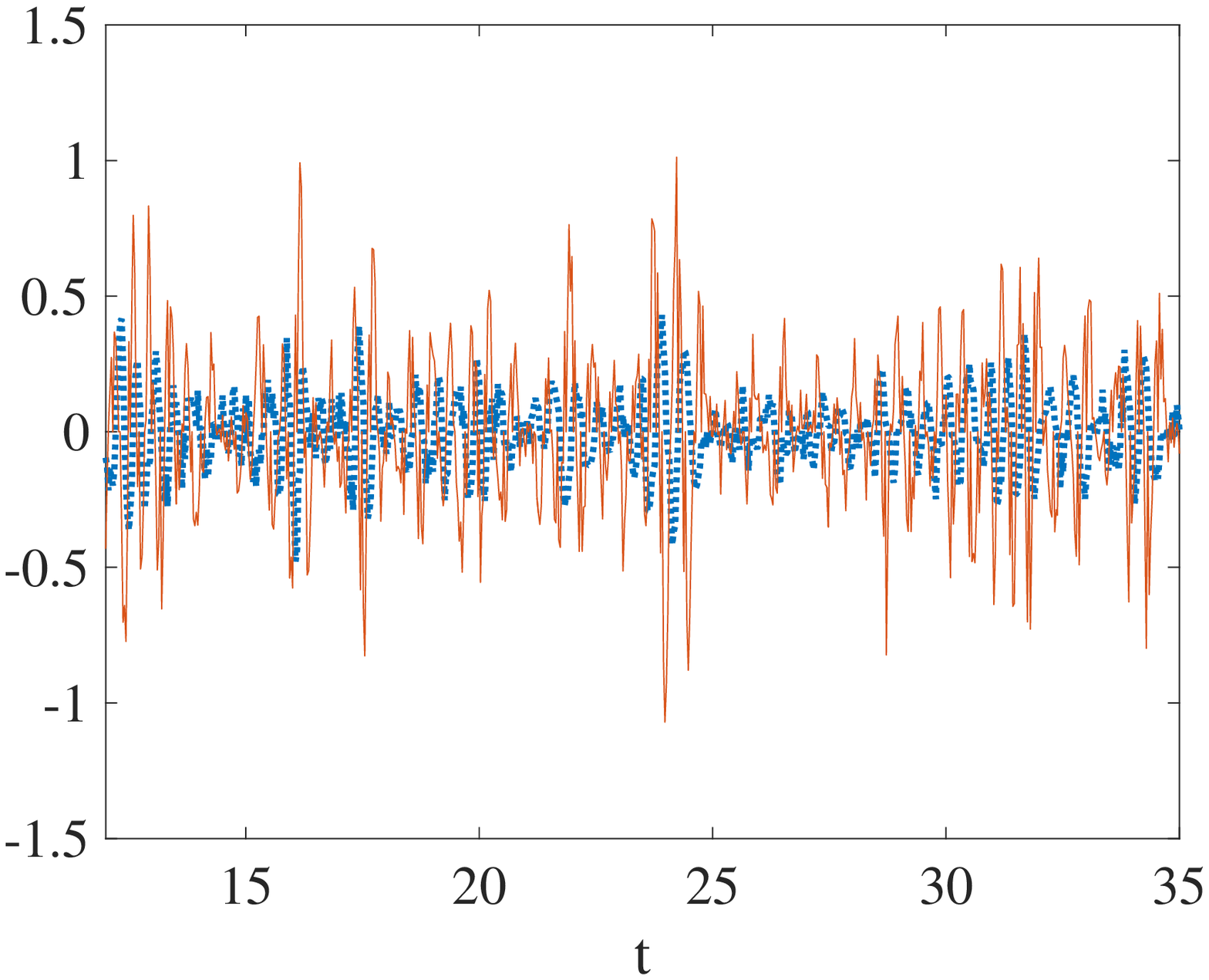}\label{fig:exp_W2}}}\centerline{
		 \subfigure[$\frac{dx_{3}}{dt}$, $v_{d_{1}}$, $\hat{v}_{3}$]{
		\includegraphics[width=0.5\columnwidth]{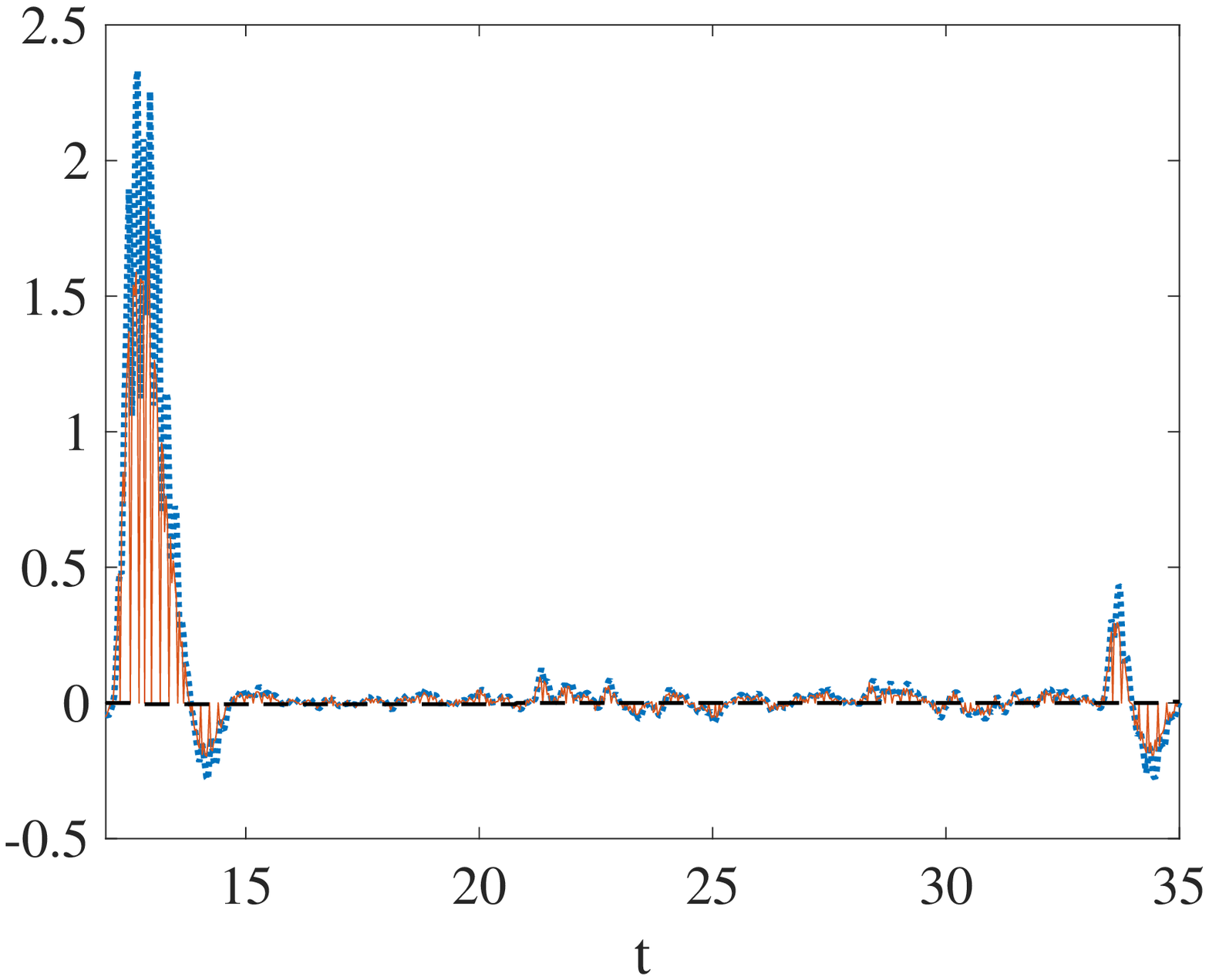}\label{fig:exp_v3}}
		 \subfigure[$\Omega_{3}$, $\hat{\Omega}_{3}$]{
		\includegraphics[width=0.5\columnwidth]{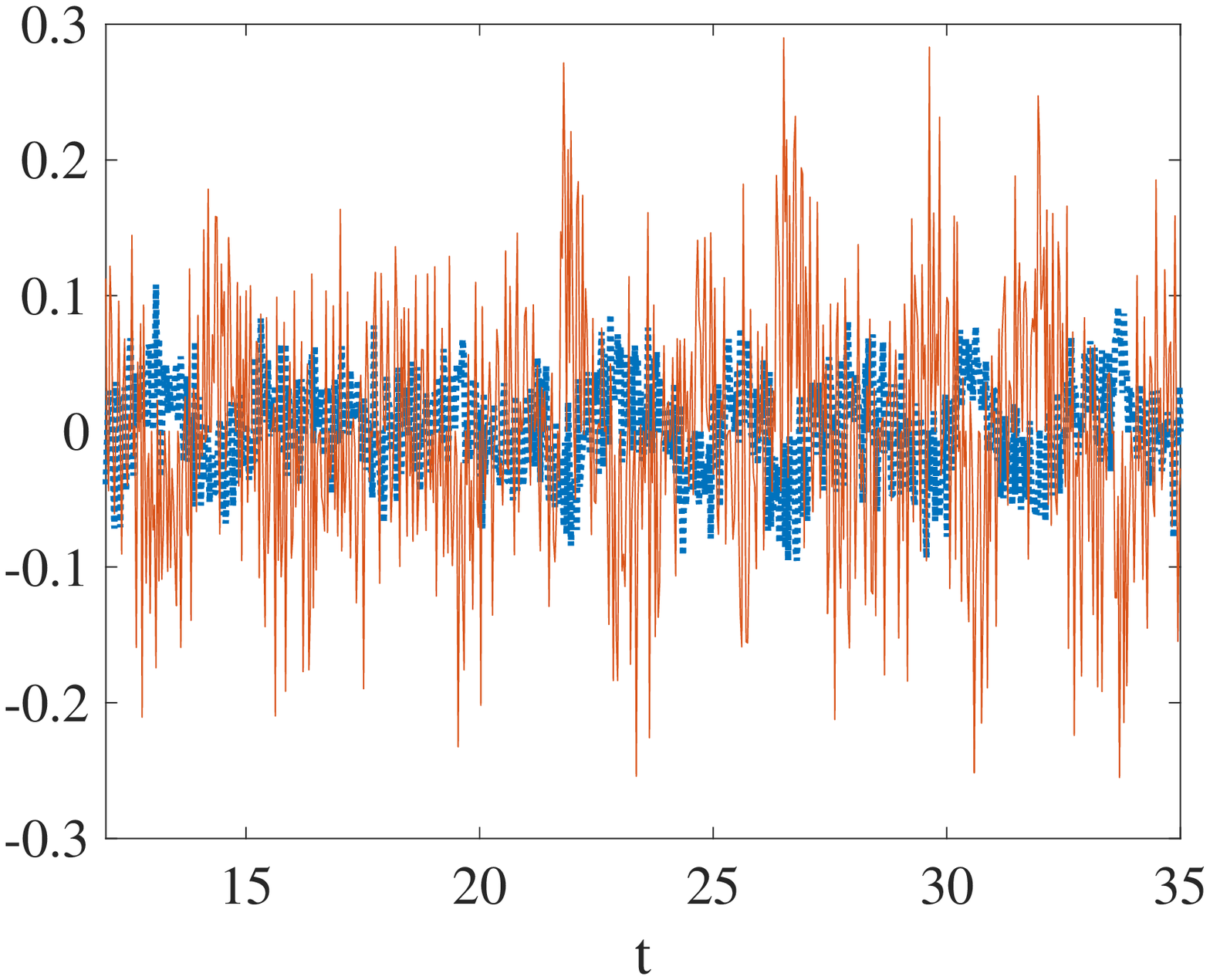}\label{fig:exp_W3}}
}
\caption{EKF performance in real-time experiment (dotted: EKF, dashed: desired, solid: noisy measurements).}\label{fig:exp}
\end{figure}
Quadrotor parameters are same as presented in the numerical simulations section and controller gains are tuned to $k_{x}=4.0$, $k_{v}=2.0$, $k_{R}=0.62$, $k_{\Omega}=0.15$, and $k_{I}=0.1$ with $c_{1}=c_{2}=0.1$. The observation matrix, $H\in\Re^{9\times18}$ utilized in this example same as one presented in the first numerical simulation and so for measurement state, $\zb_{k}\in\Re^{9\times1}$.

Figure \ref{fig:exp_x} illustrates the position of the quadrotor following the desired trajectory obtained from the experiment. The quadrotor starts from the ground, increases its altitude to get to $z_3=-0.3\ \mathrm{m}$, and starts following the trajectory. Quadrotor attitude errors, $\Psi$, $e_{R}$, and angular velocity error, $e_{\Omega}$, which are previously defined in Eq. \refeqn{psii}, \refeqn{errr}, and \refeqn{eomega}, respectively, are presented in Figure \ref{fig:exp_error}, which illustrate the performance of the geometric nonlinear controller during the autonomous trajectory tracking in this test. 
Figures \ref{fig:exp_v1}, \ref{fig:exp_v2}, and \ref{fig:exp_v3} illustrate the estimated velocity obtained form EKF namely, $\hat{v}$ and plotted verse the desired velocity, $v_{d}$. We have also presented the integrated velocity, $\frac{dx}{dt}$ in this figure by taking direct derivative from the position measurements to magnify the performance of the filter. Figures \ref{fig:exp_W1}, \ref{fig:exp_W2}, and \ref{fig:exp_W3} are also present the angular velocity measurements from IMU, and the estimated angular velocity from the EKF, $\hat{\Omega}$. As it is clear from the figures, EKF is able to improve the measurement noises significantly and provide more smoother data for real-time experimentation. The effect of noise from velocity measurements is substantially reduced with the extended Kalman filter, while avoiding any lagging issues common with low-pass filtering. 

%
\section{CONCLUSIONS}
An extensive Extended Kalman Filter on $\SE$ for full nonlinear dynamic model of quadrotor UAV employing geometric nonlinear controller presented in this paper. Linearization performed considering all the coupling effects between translational and rotational dynamics to provide an precise estimation characteristics in different scenarios of sensor failure or receiving highly noised measurements. These results validated by numerical simulations followed by experimental results for velocity estimation during an autonomous trajectory tracking.

\section*{Acknowledgment}
\addcontentsline{toc}{section}{Acknowledgment}
The authors would like to acknowledge Evan Kaufman for his helpful comments that served to improve the clarity of the discussion and writing of this paper.

\bibliographystyle{IEEEtran}   
\bibliography{ICUAS16}  


\appendix

\subsection{Proof for Proposition \ref{prop:linearizedcontroled}}\label{sec:controllinear}
We use the definition and properties of time derivative and variation to linearize the closed loop system. The variation of time derivative of each state in the nonlinear equations can be presented as follow
\begin{gather}
\delta \dot{x}=\delta v,\label{eqn:ddx}\\
\delta{\dot{v}}=m_{21}\delta x+m_{22}\delta v +m_{23}\eta+m_{24}\delta e_{i1},\label{eqn:ddv}\\ 
\dot{\eta}=\delta\Omega-\hat{\Omega}\eta,\label{eqn:deta}
\end{gather}
\begin{align}
\delta\dot{\Omega}=&m_{41}\delta x+m_{42}\delta v+m_{43}\eta\nonumber\\
&+m_{44}\delta\Omega+m_{45}\delta e_{i1}-k_{I}J^{-1}\delta e_{i2},\label{eqn:ddW}
\end{align}
\begin{gather}
\delta \dot{e}_{i1}=c_1\delta e_{x}+\delta e_{v}=c_1\delta{x}+\delta{v},
\end{gather}
\begin{align}
\delta \dot{e}_{i2}=&c_2\delta e_{R}+\delta e_{\Omega}\nonumber\\
&=m_{61}\delta{x}+m_{62}\delta{v}+m_{63}\delta{\eta}+m_{64}\delta{\Omega}+m_{65}\delta{e_{i1}}.
\end{align}
where all the sub-matrices $m_{ij}$ in the above expressions are presented in the following sections. These equations can be written in a matrix form as
\begin{align}
\dot{\xb}=A_{L}\xb,
\end{align}
where the state vector $\xb\in\Re^{18\times1}$ is given by
\begin{align*}\label{eqn:lineqn}
\xb=\begin{bmatrix}
\delta{x},\delta v,\eta,\delta\Omega,\delta{e}_{i1},\delta{e}_{i2}\\
\end{bmatrix}^{T},
\end{align*}
and matrix $A_k\in\Re^{18\times 18}$
\begin{align}
A_{L}=\begin{bmatrix}
0&I&0&0&0&0 \\
m_{21}&m_{22}&m_{23}&0&m_{24}&0\\
0&-\hat{\Omega}&I&0&0&0\\
m_{41}&m_{42}&m_{43}&m_{44}&m_{45}&-k_{I}J^{-1}\\
c_1&I&0&0&0&0 \\
m_{61}&m_{62}&m_{63}&m_{64}&m_{65}&0 
\end{bmatrix}.
\end{align}
\subsection{Proof for Equation \refeqn{ddv}} 
By taking time-derivative of both sides of Eq.~\refeqn{EL2}, we would have
\begin{align}\label{eqn:deltaxdot}
m\delta{\dot{v}}=-\delta{f}Re_{e}-f\delta{R}e_{3}=-\delta{f}Re_{e}-fR\hat{\eta}e_{3},
\end{align}
using \refeqn{f} and
\begin{align}\label{eqn:deltaA}
\delta{A}=-k_{x}\delta{e}_{x}-k_{v}\delta{e}_{v},
\end{align}
we obtain
\begin{align}
\delta f=-(-k_{x}\delta x-k_{v}\delta v)\cdot R3_{e}-A\cdot R\hat{\eta}e_{3}.
\end{align}
Substituting the above equations into \refeqn{deltaxdot} and simplifying
\begin{align}\label{eqn:deltav}
\delta{\dot{v}}=m_{21}\delta x+m_{22}\delta v +m_{23}\eta+m_{24}\delta e_{i1},
\end{align}
where $m_{21},m_{22},m_{23},m_{24}\in\Re^{3\times3}$ are defined as  
\begin{align}
m_{21}=-\frac{k_{x}}{m}(Re_{3})(Re_{3})^{T},\;m_{22}=\frac{k_{v}}{k_{x}}m_{21},\; m_{24}=\frac{k_{I}}{k_{x}}m_{21},
\end{align}
\begin{align}
m_{23}=-\frac{1}{m}(Re_{3}A^{T}R\hat{e}_{3}+(A\cdot Re_{3})R\hat{e}_{3}),
\end{align}
\subsection{Proof for Equation \refeqn{deta}}
Taking derivative of Eq.~\refeqn{EL3}, and using Eq.~\refeqn{delR}, we obtain
\begin{align}
R\hat{\Omega}\hat{\eta}+R\hat{\dot{\eta}}=R\hat{\eta}\hat{\Omega}+R\delta\hat{\Omega}.
\end{align}
Multiplying both side of the above expression by $R^{-1}$, using $\hat{x}\hat{y}-\hat{y}\hat{x}=(xy)^\wedge$ and simplifying, Eq. \refeqn{deta} is observed.
\subsection{Proof for Equation \refeqn{ddW}}
We first find the variation of \refeqn{errr} and \refeqn{eomega}, considering \refeqn{M} and simplify the equations utilizing $A^{T}\hat{x}+\hat{x}A=\trs{A}I-A$, to obtain
\begin{align}\label{eqn:eWd}
\delta e_{\Omega}=\delta\Omega-(R^{T}R_{d}\Omega_{d})^{\wedge}\eta+R^{T}R_{d}\hat{\Omega}_{d}\eta_{d}-R^{T}R_{d}\delta\Omega_{d},
\end{align}
\begin{align}\label{eqn:eRd}
\delta e_{R}=\frac{1}{2}(\trs{R^{T}R_{d}}I-R^{T}R_{d})\eta-\frac{1}{2}(\trs{R_{d}^{T}R}I-R_{d}^{T}R)\eta_{d}.
\end{align}
Then, we need to find the variation for $R_{d}$, $\Omega_{d}$ and their time-derivatives which are summarized below in each section. We utilize these derivations to find an explicit expression for the variation of Eq.~\refeqn{EL4} represented with the linearized state variables.
\subsubsection{Variation of $R_d$}
The desired attitude $R_d$ is defined as follows
\begin{align*}
R_d = [b_{1c},\, b_{2c},\, b_{3c}]
=[b_{2c}\times b_{3c},\, \frac{b_{3c}\times b_{1d}}{\|b_{3c}\times b_{1d}\|},\, b_{3c}],
\end{align*}
where
\begin{align*}
\delta b_{3c}=-\frac{\delta{A}}{\|A\|}+\frac{A(A\cdot \delta{A})}{\|A\|^{3}} = b_{3c} \times (b_{3c} \times \frac{\delta A}{\|A\|}).
\end{align*}
Using Eq. \refeqn{deltaA} and simplifying
\begin{align*}
\delta b_{3c} = z_3\times b_{3c},\; z_{3}=-b_{3c} \times \frac{\delta A}{\|A\|},
\end{align*}
where $z_{3}\in\Re^{3\times1}$. Similarly, the variation of $b_{2c}$ and $b_{1c}$ are given by
\begin{gather}
\delta b_{2c} = z_2\times b_{2c},\; \delta b_{1c} = z_1\times b_{1c},
\end{gather}
where vectors $z_1, z_2\in\Re^{3\times1}$ are defined as 
\begin{gather}
z_1 = (b_{3c}\cdot z_2) b_{3c} + ( b_{2c}\cdot z_3) b_{2c},\\
z_2 = \frac{b_{2c}\times ((z_3\times b_{3c}) \times b_{1d})}{\|b_{3c}\times b_{1d}\|}.
\end{gather}
From the above derivations, variation of $R_d$ is given by
\begin{align}\label{eqn:deltard}
\delta R_d = [z_1\times b_{1c},\, z_2\times b_{2c},\, z_3\times b_{3c}].
\end{align}
As $R_d\in\SO$, the variation of $R_d$ can be written as $\delta R_d = R_d\hat\eta_d$ for $\eta_d\in\Re^{3\times1}$. Thus, using \refeqn{deltard} and simplifying, $\eta_d$ is given by
\begin{align*}
\eta_d = (R_d^T\delta R_d)^{\vee} =a_{1}z_{3},
\end{align*}
where $a_{1}\in\Re^{3\times 3}$ is
\begin{align}
a_{1}=[b_{1c}^T,\; b_{2c}^T,\; \frac{(b_{1d}\cdot b_{3c})}{\|b_{3c}\times b_{1d}\|^2}b_{1d}^T]^{T}.
\end{align}
\subsubsection{Variation of $\Omega_{d}$}
We use the angular velocity equation $\hat{\Omega}_{d}=R_{d}^{T}\dot{R}_{d}$ and take variation from both sides. First we find an expression for the time derivative of $R_{d}$ which is given as
\begin{align*}
\dot{R}_d = [\dot{b}_{1c},\, \dot{b}_{2c},\, \dot{b}_{3c}].
\end{align*}
Similar to the pervious steps we can find an expression for 
\begin{gather}
\dot{b}_{3c}=\zeta_{3}\times b_{3c},\; \dot{b}_{2c}=\zeta_{2}\times b_{2c},\; \dot{b}_{1c}=\zeta_{1}\times b_{1c},
\end{gather}
where $\zeta_{1},\zeta_{2},\zeta_{3}\in\Re^{3\times1}$ are
\begin{gather}
\zeta_{3}=-b_{3c}\times\frac{\dot{A}}{\|A\|},\\
\zeta_{2}=\frac{b_{2c}\times((\zeta_{3}\times b_{3c})\times b_{1d}+b_{3c}\times\dot{b}_{1d})}{\|b_{3c}\times b_{1d}\|},\\
\zeta_{1}=(b_{3c}\cdot\zeta_{2})b_{3c}+(b_{2c}\cdot\zeta_{3})b_{2c},
\end{gather}
thus
\begin{align*}\label{eqn:rd_dot}
\dot{R}_d = [\dot{b}_{1c},\, \dot{b}_{2c},\, \dot{b}_{3c}]=[\zeta_{1}\times b_{1c},\ \zeta_{2}\times b_{2c},\ \zeta_{3}\times b_{3c}].
\end{align*}
Substituting the above expression into $\hat{\Omega}_{d}=R_{d}^{T}\dot{R}_{d}$ and simplifying, $\Omega_d$ can be written as
\begin{align*}
\Omega_d = a_{1}\zeta_3+a_{2},
\end{align*}
where $a_{2}\in\Re^{3\times1}$
\begin{align}
a_{2}=[0,\;0,\; \frac{\dot{b}_{1d}\cdot b_{2c}}{\|b_{3c}\times b_{1d}\|}]^{T}.
\end{align}
Time-derivative of $\eta_{d}$ can be presented as
\begin{align}\label{eqn:etadot}
\dot{\eta}_{d}=\dot{a}_{1}z_{3}+a_{1}\dot{z}_{3}.
\end{align} 
After simplifying, we can present $\dot{z}_{3}$ as
\begin{align*}
\dot{z}_{3}=B_{1}\delta x+B_{2}\delta v+B_{3}\eta+B_{4}\delta e_{i1},
\end{align*}
where sub-matrices $B_{1},B_{2},B_{3},B_{4}\in\Re^{3\times3}$ are
\begin{align*}
B_{1}=-k_{x}X_{1}-k_{v}X_{2}m_{21}-k_{I}B_{5}X_{2},
\end{align*}
\begin{align*}
B_{2}=-k_{v}X_{1}-(k_{x}+k_{I})X_{2}-k_{v}X_{2}m_{22},
\end{align*}
\begin{align*}
B_{3}=-k_{v}X_{2}m_{23},\; B_{4}=-k_{I}X_{1}-k_{v}X_{2}m_{24},
\end{align*}
and $X_{1},X_{2}\in\Re^{3\times3}$
\begin{align*}
X_{1}=-\frac{(\dot{b}_{3c})^{\wedge}}{\|A\|}+\frac{({b}_{3c})^{\wedge}(A\cdot\dot{A})}{\|A\|^{3}}, \; X_{2}=-\frac{(b_{3c})^{\wedge}}{\|A\|}.
\end{align*}
Substituting the time derivative of $a_{1}$ and the above expressions into Eq. \refeqn{etadot}
\begin{align*}
\dot{\eta}_{d}=B_{5}\delta x+B_{6}\delta v+B_{7}\eta+k_{I}B_{9}\delta e_{i1},
\end{align*}
where $B_{5},B_{6},B_{7},B_{8},B_{9}\in\Re^{3\times3}$ are given by
\begin{align*}
B_{5}=\frac{k_{x}\dot{a}_{1}(b_{3c})^{\wedge}}{\|A\|}+a_{1}B_{1},\; B_{6}=\frac{k_{v}\dot{a}_{1}(b_{3c})^{\wedge}}{\|A\|}+a_{1}B_{2},
\end{align*}
\begin{align*}
B_{7}=a_{1}B_{3},\; B_{8}=\frac{k_{I}\dot{a}_{1}(b_{3c})^{\wedge}}{\|A\|}+a_{1}B_{4},\; B_{9}=\frac{(\hat{\Omega}_{d}a_{1}(b_{3c}^{\wedge}))}{\|A\|}.
\end{align*}
After simplifying, the variation of $\Omega_{d}$ is given by
\begin{align}\label{eqn:omegddd}
\delta\Omega_{d}=&\dot{\eta}_{d}+\hat{\Omega}_{d}\eta_{d}\nonumber\\
=&(B_{5}+k_{x}B_{9})\delta x+(B_{6}+k_{v}B_{9})\delta v\nonumber\\
&+B_{7}\eta+(B_{8}+k_{I}B_{9})\delta e_{i1}.
\end{align}
\subsubsection{Time derivative of $\delta{\Omega}_{d}$}
Taking time derivative of Eq.~\refeqn{omegddd}, using $\dot{\eta}=\delta\Omega-\hat{\Omega}\eta$, and Eq. \refeqn{deltav}, we have
\begin{align*}
\delta\dot{\Omega}_{d}=F_{1}\delta x+F_{2}\delta v+F_{3}\eta+B_{7}\delta\Omega+F_{4}\delta e_{i1},
\end{align*}
where matrices $F_{1},F_{2},F_{3},F_{4}\in\Re^{3\times3}$ are defined as
\begin{align*}
F_{1}=\dot{B}_{5}+k_{x}\dot{B}_{9}+(B_{6}+k_{v}B_{9})A_{21}+B_{5}(B_{8}+k_{I}B_{9}),
\end{align*}
\begin{align*}
F_{2}=&B_{5}+k_{x}B_{9}+\dot{B}_{6}+k_{v}\dot{B}_{9}\\
&+(B_{6}+k_{v}B_{9})A_{22}+(B_{8}+k_{I}B_{9}),
\end{align*}
\begin{align*}
F_{3}=(B_{6}+k_{v}B_{9})A_{23}+\dot{B}_{7}-B_{7}\hat{\Omega},
\end{align*}
\begin{align*}
F_{4}=\dot{B}_{8}+k_{I}\dot{B}_{9}.
\end{align*}
\subsubsection{Variations of $e_{i1}$ and $e_{i2}$}
The variation of $e_{i1}$ and $e_{i2}$ can be expressed as
\begin{align*}
\delta e_{i1}=c\delta e_{x}+\delta e_{v}=c\delta{x}+\delta{v},
\end{align*}
\begin{align*}
\delta e_{i2}&=c\delta e_{R}+\delta e_{\Omega}\nonumber\\
&=m_{61}\delta{x}+m_{62}\delta{v}+m_{63}\delta{\eta}+m_{64}\delta{x\Omega}+m_{65}\delta{e_{i1}},
\end{align*}
where $m_{61},m_{62},m_{63},m_{64},m_{65}\in\Re^{3\times3}$ are defined as
\begin{align*}
m_{61}=-B_{5}G_{5}Y_{1}+G_{2}Y_{1}-G_{3}(B_{5}+k_{x}B_{9}),
\end{align*}
\begin{align*}
m_{62}=-B_{5}G_{5}Y_{2}+G_{2}Y_{1}-G_{3}(B_{6}+k_{v}B_{9}),
\end{align*}
\begin{align*}
m_{63}=B_{5}G_{4}-G_{1}-G_{3}B_{7},\; m_{64}=I_{3\times 3},
\end{align*}
\begin{align*}
m_{65}=-B_{5}G_{5}Y_{3}+G_{2}Y_{3}-G_{3}(B_{8}+k_{I}B_{9}),
\end{align*}
and matrices $G_{1},G_{2},G_{3},G_{4},G_{5}\in\Re^{3\times3}$ are 
\begin{align*}
G_{1}=(R^{T}R_{d}\Omega_{d})^{\wedge}, G_{2}=R^{T}R_{d}\hat{\Omega_{d}}, G_{3}=R^{T}R_{d},
\end{align*}
\begin{align*}
G_{4}=\frac{1}{2}(\trs{R^{T}R_{d}}I-R^{T}R_{d}),
\end{align*}
\begin{align*}
G_{5}=\frac{1}{2}(\trs{R_{d}^{T}R}I-R_{d}^{T}R).
\end{align*}
\subsubsection{Variation for Equation \refeqn{EL4}}
Substituting the above derivations for $\delta\Omega_{d}$, $\delta R_{d}$ into Eq.~\refeqn{eWd}, and \refeqn{eRd} and simplifying, we can present the following expression for Eq.~\refeqn{EL4} 
\begin{align}
\delta\dot{\Omega}=&m_{41}\delta x+m_{42}\delta v+m_{43}\eta+m_{44}\delta\Omega\nonumber\\
&+m_{45}\delta e_{i1}-k_{I}J^{-1}\delta e_{i2},
\end{align}
where $m_{41},m_{42},m_{43},m_{44},m_{45},\in\Re^{3\times3}$
\begin{align}
m_{41}=B_{11}\frac{k_{x}a_{1}(b_{3c})^{\wedge}}{\|A\|}+B_{14}F_{1}+B_{13}(B_{5}+k_{x}B_{9}),
\end{align}
\begin{align}
m_{42}=B_{11}\frac{k_{v}a_{1}(b_{3c})^{\wedge}}{\|A\|}+B_{14}F_{2}+B_{13}(B_{6}+k_{v}B_{9}),
\end{align}
\begin{align}
m_{43}=B_{10}+B_{14}F_{3}+B_{13}B_{7},\;m_{44}=B_{12}+B_{14}B_{7},
\end{align}
\begin{align}
m_{45}=B_{11}\frac{k_{I}a_{1}(b_{3c})^{\wedge}}{\|A\|}+B_{13}(B_{8}+k_{I}B_{9})+B_{14}F_{4},
\end{align}
and sub-matrices $B_{10}, B_{11}, B_{12}. B_{13}, B_{14}\in\Re^{3\times3}$ are given by the following expressions
\begin{align}
B_{10}=J^{-1}[&-k_{R}G_{4}+k_{\Omega}(R^{T}R_{d}\Omega_{d})^{\wedge}\nonumber\\
&-(J R^{T}R_{d}\Omega_{d})^{\wedge}(R^{T}R_{d}\Omega_{d})^{\wedge}\nonumber\\
&+(R^{T}R_{d}\Omega_{d})^{T}J(R^{T}R_{d}\Omega_{d})^{\wedge}+J(R^{T}R_{d}\dot{\Omega}_{d})^{\wedge}],
\end{align}
\begin{align}
B_{11}=J^{-1}[&k_{R}G_{5}-k_{\Omega}R^{T}R_{d}\hat{\Omega}_{d}\nonumber\\
&+(J R^{T}R_{d}\Omega_{d})^{\wedge}R^{T}R_{d}\hat{\Omega}_{d}\nonumber\\
&-(R^{T}R_{d}\Omega_{d})^{\wedge}J R^{T}R_{d}\hat{\Omega}_{d}-J R^{T}R_{d}\hat{\dot{\Omega}}_{d}],
\end{align}
\begin{align}
B_{12}=J^{-1}[-k_{\Omega}I(J\Omega)^{\wedge}-\hat{\Omega}J],
\end{align}
\begin{align}
B_{13}=J^{-1}[&k_{\Omega}R^{T}R_{d}-(J R^{T}R_{d}\Omega_{d})^{\wedge}R^{T}R_{d}\nonumber\\
&+(R^{T}R_{d}\Omega_{d})^{\wedge}J R^{T}R_{d}],
\end{align}
\begin{align}
B_{14}=J^{-1}[JR^{T}R_{d}].
\end{align}


\end{document}